\documentclass[11pt,reqno,a4paper]{amsart}
\usepackage{hyperref}
\usepackage{amsmath}
\usepackage{amssymb,a4wide}
\usepackage{amsthm}
\usepackage[margin=3cm]{geometry}
\usepackage{graphicx,epstopdf}

\hypersetup{
    colorlinks=true,
    linkcolor=black,
    filecolor=black,      
    urlcolor=blue,
    citecolor=red,
}

\newcommand{\guio}[1]{\nobreakdash-\hspace{0pt}}

\numberwithin{equation}{section}
\newtheorem{theorem}{Theorem}[section]
\newtheorem{lemma}[theorem]{Lemma}
\newtheorem{corollary}[theorem]{Corollary}

\theoremstyle{definition}

\newtheorem{remark}[theorem]{Remark}

\newcommand{\R}{\mathbb{R}}
\newcommand{\C}{\mathbb{C}}

\newcommand{\re}{\operatorname{Re}}
\newcommand{\im}{\operatorname{Im}}
\newcommand{\pv}{\operatorname{p.\!v.}}
\newcommand{\arctanh}{\operatorname{arctanh}}

\newcommand{\ep}{\varepsilon}

\newcommand{\supp}{\operatorname{supp}}

\title[Stability of ellipsoids as minimisers of perturbed Coulomb energies]
{Stability of ellipsoids as the energy minimisers \\of perturbed Coulomb energies}
\author[J.~Mateu]{J.~Mateu}
\author[M.G.~Mora]{M.G.~Mora}
\author[L.~Rondi]{L.~Rondi}
\author[L.~Scardia]{L.~Scardia}
\author[J.~Verdera]{J.~Verdera}

\address[J. Mateu]{Department de Matem\`atiques, Universitat Aut\`onoma de Barcelona, and Centre de Recerca Matem\`atica, Catalonia}
\email{mateu@mat.uab.cat}

\address[M.G. Mora]{Dipartimento di Matematica, Universit\`a di Pavia, Italy}
\email{mariagiovanna.mora@unipv.it}

\address[L. Rondi]{Dipartimento di Matematica, Universit\`a di Pavia, Italy}
\email{luca.rondi@unipv.it}

\address[L. Scardia]{Department of Mathematics, Heriot-Watt University, United Kingdom}
\email{L.Scardia@hw.ac.uk}

\address[J. Verdera]{Department de Matem\`atiques, Universitat Aut\`onoma de Barcelona, and Centre de Recerca Matem\`atica, Catalonia}
\email{jvm@mat.uab.cat}

\date{}

\begin{document}

\begin{abstract} 
In this paper we characterise the minimiser for a class of nonlocal perturbations of the Coulomb energy. We show that the minimiser is the normalised characteristic function of an ellipsoid, under the assumption that the perturbation kernel has the same homogeneity as the Coulomb potential, is even, smooth off the origin and sufficiently small. This result can be seen as the stability of ellipsoids as energy minimisers, since the minimiser of the Coulomb energy is the normalised characteristic function of a ball. 
\end{abstract}

\maketitle

\section{Introduction and statement of the main result}
Nonlocal energies are an approximation of discrete energies modelling long-range particle interactions, for large numbers of particles. The study of the minimisers of nonlocal energies -- existence, uniqueness, regularity and characterisation -- is therefore a crucial step for understanding optimal arrangements of particles, at least in average. 

In this paper we characterise the minimisers for a class of nonlocal energies that are perturbations of the Coulomb energy. We focus here on the two-dimensional case to illustrate the main result and the key ideas of our approach.

We consider energy functionals $I^\kappa$ defined on probability measures $\mu \in \mathcal{P}(\R^2)$ as
\begin{equation}\label{en0}
I^\kappa(\mu)=\int_{\R^2}\int_{\R^2} W^\kappa(z-w)\, d\mu(z) d\mu(w) + \int_{\R^2} |z|^2 \,d\mu(z),
\end{equation}
where the interaction potential $W^\kappa$ is given (in complex variables) by
\begin{equation}\label{intpot}
W^\kappa(z) = -\log|z|+\kappa(z), \quad z \in \C, \quad z\neq0,
\end{equation}
and $W^{\kappa}(0)=+\infty$, and $\kappa$ is an even real-valued function, homogeneous of degree $0$ and of class $C^3(\C\setminus\{0\}).$ 

The unperturbed energy $I^0$, defined in \eqref{en0}, for $\kappa=0$, is perhaps the most well-studied nonlocal energy, due to its relevance in a variety of contexts, from random matrices to interpolation theory and materials science. 
The minimiser of $I^0$ is well-known, and is the normalised characteristic function of the unit disc, the so-called circle law (see \cite{Fro}, \cite{ST}, and the references therein). 

The main result of this paper is that the perturbed energy $I^\kappa$ also has a unique minimiser, which is the normalised characteristic function of an ellipse, provided the kernel $\kappa$ is small in some suitable norm. This can be seen as a `stability' result for ellipses, showing the `persistence' of the ellipse as the energy minimiser, for small perturbations of the energy.

\subsection{Motivation} In the recent work \cite{CMMRSV, MMRSV20, MMRSVIzv} we considered the one-parameter family of energies 
\begin{equation}\label{enal}
I_\alpha(\mu)=\int_{\R^2}\int_{\R^2} W_\alpha(z-w)\, d\mu(z) d\mu(w) + \int_{\R^2} |z|^2 \,d\mu(z),
\end{equation}
defined on $\mu \in \mathcal{P}(\R^2)$, where the interaction kernel  $W_\alpha$ is given by 
\begin{equation*}
W_\alpha(z) = -\log|z|+\alpha \frac{x^2}{|z|^2}, \quad z=x+i y \in \C, \quad z\neq0, \quad -1< \alpha<1.
\end{equation*}
The energy $I_\alpha$ arises in the study of defects in metals, dislocations, in the limit case $\alpha=1$ (see \cite{MRS}). 
In \cite{CMMRSV} we showed that the minimiser of $I_\alpha$ is the normalised characteristic function of the region encircled by an ellipse with semi-axes $\sqrt{1-\alpha}$ and $\sqrt{1+\alpha}$ for every $-1<\alpha<1$ (see also \cite{CMMRSV2} for a higher-dimensional version of the result).

For $\alpha$ small, we can interpret the energy $I_\alpha$ in \eqref{enal} as a `perturbation' of the Coulomb energy $I^0$. Hence the minimality of the ellipse for $I_\alpha$, for $\alpha$ small, shows the persistence of the ellipse as energy minimiser when the logarithmic potential is perturbed by $\alpha x^2/|z|^2$. 

A natural question is then what is special about the potential $\alpha x^2/|z|^2$, and whether we can {\em reproduce the persistence of the ellipse} for more general perturbations $\kappa$ of the logarithmic potential. This is one of the main motivations of this work. In Theorem~\ref{te} we identify the properties of the perturbation potential $\kappa$ that guarantee the persistence of the ellipse: if $\kappa$ is even, zero-homogeneous, and smooth outside the origin, then the minimiser of the corresponding energy is still an ellipse, at least if $\kappa$ is sufficiently small.

Another motivation for our study comes from applications in materials science, where kernels of the form \eqref{intpot} arise in the study of dislocations in anisotropic elastic media. For instance, the interaction of screw dislocations in a planar anisotropic elastic body is described in terms of the kernel
$$
-\frac12\log \left(\alpha x^2-2\beta xy+\gamma y^2\right), \quad z=x+i y \in \C, \quad z\neq0,
$$
where $\alpha,\beta,\gamma$ are given constants such that $\alpha>0$ and $\alpha\gamma-\beta^2>0$. This kernel can be written in the form \eqref{intpot} by considering
$$
\kappa(z)=-\frac12\log \left(\frac{\alpha x^2-2\beta xy+\gamma y^2}{|z|^2}\right).
$$
Similarly, the interaction of edge dislocations in a planar anisotropic elastic body involves a kernel of the form \eqref{intpot} with
$$
\kappa(z)=-\frac14\log \left(\frac{x^4+(2\alpha+\beta)x^2y^2+\alpha^2y^4}{|z|^4}\right)
-\frac{\sqrt{4\alpha+\beta}}{2\sqrt\beta}\arctanh\left(\frac{2x^2+(2\alpha+\beta)y^2}{y^2\sqrt{4\alpha\beta+\beta^2}}\right),
$$
where $\alpha,\beta$ are given constants such that $\beta>0$ and $4\alpha+\beta>0$ (see, e.g., \cite[Chapter~13]{HL}).

\subsection{Main result.}
Before stating our main result we fix some notation. Given positive real numbers $a$ and $b$ and an angle $\varphi\in [0,2\pi)$, we let $E(a,b, \varphi)$ stand for the compact set enclosed by the ellipse with semi-axes $a$ and $b$, tilted by an angle $\varphi$ with respect to the $x$-axis, namely
 \begin{equation}\label{ellrot}
E(a,b, \varphi)= e^{i \varphi} \,\left\{(x,y)\in \R^2: \frac{x^2}{a^2}+\frac{y^2}{b^2} \le 1\right\}.
\end{equation}
 If $\varphi=0$, we use the notation  
 \begin{equation}\label{el}
E_0(a,b):= E(a,b,0)= \left\{(x,y)\in \R^2: \frac{x^2}{a^2}+\frac{y^2}{b^2} \le 1\right\}.
\end{equation}
We will often refer to these sets as the `interior' of the boundary ellipse, interior not having here the usual topological meaning.

We are now ready to state our main result.

\begin{theorem}\label{te}
There exists $\ep_0>0$ such that if $\kappa$ is an even real function, homogeneous of degree $0,$ of class $C^3$ off the origin, and satisfies the smallness condition
\begin{equation}\label{small}
 |\nabla^j \kappa(z)| \le \ep_0 \quad \textrm{for } \, |z|=1 \quad \textrm{and }\, j \in \{0,1,2,3\},
\end{equation}
then there exists a triple $(a,b,\varphi)$ such that the probability measure $\chi_E /|E|$, with $E =E(a,b,\varphi)$ as in \eqref{ellrot}, is the unique minimiser of the energy \eqref{en0}.
\end{theorem}

The result in Theorem~\ref{te} is threefold: it gives existence, uniqueness and characterisation of the minimiser of the energy \eqref{en0}. While the existence is quite standard, uniqueness and characterisation are more subtle. 

To prove \textit{uniqueness}, we show that the energy $I^\kappa$ is strictly convex on a class of measures that are relevant for the minimisation. We achieve this by showing that the Fourier transform of the potential $W^\kappa$ is positive outside zero. Note that $W^\kappa=W^0+\kappa$, where $W^0=-\log|\cdot|$ is the logarithmic potential. For $W^0$ it is known that
\begin{equation*}
\widehat{W^0}(\xi)= 2\pi \pv \frac{1}{|\xi|^2}, \quad \xi\neq 0,
\end{equation*}
so clearly $\widehat{W^0}(\xi)>0$ for $\xi \neq 0$. In Section~\ref{FT} we show that the assumptions on $\kappa$ ensure that $\hat \kappa$ has a similar structure as $\widehat{W^0}$, and that adding $\hat\kappa$ to $\widehat{W^0}$ does not disrupt its positivity outside the origin. 

For the \textit{characterisation} of the minimiser, we use the Euler-Lagrange conditions
\begin{align}\label{in:EL1}
 \left(W^\kappa\star \mu  \right) (z)+ \frac{1}{2}|z|^2&= C \quad \text{for }\mu\text{-a.e. } z \in \supp\mu,\\
 \left(W^\kappa\star \mu  \right) (z)+ \frac{1}{2}|z|^2&\ge C \quad \text{for }z \in \R^2 \setminus N \text{ with } \operatorname{Cap}(N)=0, \label{in:EL2}
\end{align}
where $\supp\mu$ stands for the support of $\mu$, $C$ is a constant, and $\operatorname{Cap}$ is the logarithmic capacity.
Due to the strict convexity of the energy, these conditions are equivalent to minimality, and are satisfied by the unique minimiser of the energy only.  

For the first Euler-Lagrange condition \eqref{in:EL1}, our approach is to impose that $\chi_E /|E|$, for a generic ellipse $E$ as in \eqref{ellrot}, satisfies it. We recall that $-\log|\cdot| \star \chi_E$ is quadratic on $E$ for any ellipse, see, e.g., \cite{HMV}, where the potential has been computed explicitly in the context of Kirchhoff ellipses in fluid dynamics. Hence \eqref{in:EL1} can only be satisfied if also $\kappa\star \chi_E$ is quadratic on $E$. We prove that this is indeed the case, by showing that the convolution against $\partial_{ij}\kappa$ defines a special Calder\'on-Zygmund operator, which has the property of being constant on $E$, when evaluated on $\chi_E$. In other words, the assumptions on $\kappa$ (notably without the smallness condition) guarantee that 
$$
\partial_{ij} \left(\left(W^\kappa_{\alpha}\star \mu  \right) (z)+ \frac{1}{2}|z|^2\right) = \text{constant} \quad \text{on } E, 
$$
for $i,j=1,2$. Imposing that the constant is zero (as derived from \eqref{in:EL1}) gives a system of three equations (for the derivatives $\partial_{11}$, $\partial_{12}$ and $\partial_{22}$) in three unknowns (the semi-axes $a$ and $b$ and the tilting angle $\varphi$). We show that this system admits a unique solution for $\kappa$ small by resorting to a non-trivial application of the Implicit Function Theorem. 

For the second Euler-Lagrange condition \eqref{in:EL2}, instead, we adopt a purely perturbative argument, which exploits the `closeness' of every term of the equation, for $\kappa$ small, to the corresponding term of the second Euler-Lagrange condition for the case $\kappa=0$.

\begin{remark}[Tilting angle $\varphi$]\label{rem:00}
For kernels that are even in each variable separately, our proof yields $\varphi=0$ for the minimising ellipse.
The simplest case in which the minimising ellipse has a rotation angle $\varphi\neq 0$ is 
\begin{equation*}
 \kappa(z) =\beta \,\frac{xy}{|z|^2},
\end{equation*}
as was shown in  \cite{CMMRSV}. To see this, we express $\kappa$ in terms of the rotated variables $w:=(u,v) = e^{-i \pi/4} z$ and obtain the kernel 
$$
\tilde{\kappa}(w) = \frac{\beta}{2} \,\frac{u^2-v^2}{|w|^2},
$$
where $\tilde{\kappa}(w):= \kappa(z(w))$. Since $\tilde{\kappa}(w) = \beta \frac{u^2}{|w|^2} - \frac\beta2$, this kernel yields the minimisation problem with interaction potential 
$$
 -\log|w|+\beta\, \frac{u^2}{|w|^2},
$$
for which we know that minimisers are normalised characteristic functions of domains enclosed by ellipses with angle $\varphi=0$ and semi-axes $\sqrt{1-\beta}$ and $\sqrt{1+\beta}$, for $\beta\in (-1,1)$ 
(see \cite{CMMRSV}). Thus the unique minimiser
for the energy  with interaction potential
$$
-\log|z|+\beta\, \frac{xy}{|z|^2}
$$
is, for each $\beta\in (-1,1)$, the normalised characteristic function of 
$$E(\sqrt{1-\beta}, \sqrt{1+\beta}, \pi/4).$$

\end{remark}

\begin{remark}[Examples of perturbation kernels] The kernels
\begin{equation}\label{higher-anis}
\kappa(z)=\beta \frac{x^{2\ell}}{|z|^{2\ell}}, \quad \ell\in \mathbb N,
\end{equation}
are even, real-valued, homogeneous of degree $0$, and of class $C^3$ off the origin. 
Theorem~\ref{te} then guarantees that for small enough $\beta$ the corresponding energy $I^\kappa$ has a unique minimiser, which is 
the normalised characteristic function of the domain enclosed by some ellipse as in \eqref{el}, by Remark~\ref{rem:00}.
One may wonder what is the maximal interval in $\beta$ for which the minimiser is an ellipse. For $\ell=1$ we have a complete answer: the results in \cite{CMMRSV, MRS} show that for $\beta\in (-1,1)$ the energy minimiser is an ellipse, while
for $|\beta|\geq1$ it is a measure supported on a segment (the so-called semicircle law). Therefore, $(-1,1)$ is the maximal interval.
For $\ell\geq2$ the situation is unclear. For instance, for $\ell=2$ we only know from preliminary computations that for 
$$
-\frac{2}{3} < \beta < \frac{4}{3}
$$
the minimisation problem has a unique solution and there exists an ellipse solving the first Euler-Lagrange condition.
In this paper, however, we will not further investigate this possibility, and our focus will be on small perturbations $\kappa$.
\end{remark}

\subsection{Structure of the paper} 
In Section~2 we collect some  results on Calder\'on-Zygmund operators. In Section~3 we prove that the energy $I^\kappa$ admits a unique minimiser; in particular, we address the positivity of the Fourier transform of the interaction kernel. Section~\ref{sect:EL} is devoted to the proof of the existence of an ellipse satisfying the Euler-Lagrange conditions. Finally, in Section~5 we briefly discuss the higher-dimensional case.

\section{Notation and terminology} We now recall some useful results and establish some convention on notation and terminology. 

\subsection{The Fourier transform} 
The Fourier transform definition we adopt is 
$$
\widehat{\phi}(\xi) = \int_{\R^d} \phi(z) e^{-i \xi \cdot z}\, dz, \quad \xi \in \R^d,
$$
in any dimension $d\geq 1$, for functions $\phi$ in the Schwartz class.

We will use in several occasions a formula giving the Fourier transform of a distribution in $\R^d$ of the form
\begin{equation*}
\frac{\Phi_k(z)}{|z|^{d-s+k}}, \quad  z \in \R^d\setminus \{0\},
\end{equation*}
where $\Phi_k$ is a homogeneous harmonic polynomial of degree $k\geq 1$ and $0\le s \le d.$ In the case $s =0$ the above expression is understood in the principal value sense, as well as the expression on the Fourier side in the formula below for $s=d$.  The formula is 
\begin{equation}\label{foupolk}
\frac{\Phi_k(z)}{|z|^{d-s+k}}  \xrightarrow{\rm{Fourier}} (-i)^{k} 2^s \pi^{d/2} \frac{\Gamma(\frac{k+s}{2})}{\Gamma(\frac{k+d-s}{2})}\,\frac{\Phi_k(\xi)}{|\xi|^{k+s}},
\end{equation}
see \cite[Chapter~III, Theorem~5]{S} where a slightly different definition of the Fourier transform is adopted.

\subsection{Calder\'on-Zygmund operators} 
Let $T$ be an even smooth homogeneous convolution Calder\'on-Zygmund operator in $\R^d$, that is,
\begin{equation}\label{CZ-i}
T(f)(z) = \pv \int_{\R^d} f(w) H(z-w) \,dw,  \quad f \in L^2(\R^d),
\end{equation}
where $H$ is an even kernel, homogeneous of degree $-d$, of class $C^1$ off the origin and satisfying the cancellation property
\begin{equation}\label{canc}
\int_{|\xi|=1} H(\xi) \, d\sigma(\xi) =0.
\end{equation}
The above principal value integral is defined for almost all $z \in \mathbb{\R}^d.$

The Calder\'on-Zygmund constant of $T$ is defined to be 
\begin{equation}\label{CT}
 \|T\|_{\rm CZ}= \sup_{|\xi|=1} (|H(\xi)|+ |\nabla H(\xi)|).
\end{equation}
If $F\subset  \mathbb{\R}^d$, $f$ is a function defined on $F$ and $0 < \gamma<1$, we set 
$$
|f|_{\gamma,F}= \sup_{\substack{x,y \in F\\ x\neq y}} \frac{|f(x)-f(y)|}{|x-y|^{\gamma}}.
$$

The following lemma is a regularity result for the Calder\'on-Zygmund operator on smooth domains.

\begin{lemma}[\cite{MOV}]\label{intsingn}
 Let $D \subset \mathbb{\R}^d$ be a domain with boundary of class $C^{1,\gamma},$ $0 < \gamma<1,$ and $T$ an even smooth 
 homogeneous convolution Calder\'on-Zygmund operator in $\R^d.$  Then 
 \begin{equation}\label{czb}
  |T(\chi_D)(x)| \le C \, \|T\|_{\rm CZ} \quad \text{for }x \in \mathbb{\R}^d,
 \end{equation}
and
 \begin{equation}\label{czholder}
 |T(\chi_D)|_{\gamma,D}+  |T(\chi_D)|_{ \gamma, \mathbb{\R}^d \setminus \overline{D}}\le C \, \|T\|_{\rm CZ},
 \end{equation}
 for a positive constant $C$ depending on $d$, $\gamma$, and $D$. The constant $C$ depends only on the constants determining the $C^{1,\gamma}$-character of $D$.
\end{lemma}

As a consequence of Lemma~\ref{intsingn} we deduce the following result that proves the 
tangential continuity of the first order derivatives of a sort of primitive of $T(\chi_D)$.

\begin{lemma}\label{tangcont}
Let $D \subset \mathbb{\R}^d$ be a domain with boundary of class $C^{1,\gamma}$,   $0<\gamma<1$, and let $\tau(x)$ be a tangent vector to $\partial D$ at 
$x \in \partial D$. Let $G$ be an odd  kernel, homogeneous of degree $-(d-1),$ of class $C^2$ off the origin. Then the limits 
\begin{equation*}
\lim_{D \ni y \rightarrow x}	\langle \nabla G \star \chi_D (y), \tau(x) \rangle  
\qquad \text{and} \qquad \lim_{\overline{D} \not\ni y \rightarrow x}   \langle \nabla G \star \chi_D (y),  \tau(x) \rangle
\end{equation*}
exist and coincide for each $x \in \partial D.$
\end{lemma}

\begin{proof}
We compute the distributional gradient of $G$. Note that $\nabla G$ is an even kernel, with values in $\mathbb{\R}^d,$ homogeneous of degree $-d$, of class $C^1$ off the origin.

The gradient of $G$ in the  sense of distributions is a constant multiple of the Dirac delta at the origin plus the principal value distribution associated with the kernel $\nabla G$. More precisely
\begin{equation*}
\nabla G= \left(\int_{|\xi|=1} G(\xi) \xi \, d\sigma(\xi) \right) \delta_0+ \pv \nabla G.
\end{equation*}

It is a simple matter realising that 
\begin{equation}\label{Gi0}
\int_{|\xi|=1} \partial_{i} G(\xi) \, d\sigma(\xi) =0, \quad i=1,\dots, d. 
\end{equation}
To see this, we first note that, by applying the Divergence Theorem,
\begin{equation}\label{aba}
\int_{\frac{1}{2}< |z|<1} \partial_{i}G(z)\,dz = -\int_{|\xi|=\frac12}G(\xi) \nu_i(\xi)\,d\sigma(\xi)+ \int_{|\xi|=1}G(\xi) \nu_i(\xi)\,d\sigma(\xi)=0,
\end{equation}
since, by the homogeneity of $G$, the integral $\int_{|\xi|=r} G(\xi) \nu_i(\xi)\,d\sigma(\xi)$ is independent of $r>0$, where $\nu_i(\xi)$ 
is the $i$-th component of the exterior unit normal vector to the sphere centred at $0$ of radius $r$, at the point $\xi$.  Moreover, by using again the homogeneity of $G$, we conclude that
\begin{align*}
\int_{\frac{1}{2}< |z|<1} \partial_{i}G(z)\,dz&= \int_{\frac{1}{2}< |z|<1} |z|^{-d}\,\partial_{i}G\left(\frac{z}{|z|}\right)\,dz
=\int_{\frac12}^1\frac1\rho d\rho \int_{|\xi|=1}  \partial_{i} G(\xi) \, d\sigma(\xi)\\
& = \log 2 \,\int_{|\xi|=1} \partial_{i} G(\xi) \, d\sigma(\xi),
\end{align*}
and hence, by \eqref{aba}, prove \eqref{Gi0}. 

This shows that $\partial_{i}G$ is the kernel of an even homogeneous Calder\'on-Zygmund 
operator to which one can apply Lemma~\ref{intsingn}.
Therefore, $G\star\chi_D$ is a Lipschitz function in $\mathbb{\R}^d$ and 
$\nabla G\star \chi_D$ satisfies a H\"older condition of order $\gamma$ in $D$ and in $\mathbb{\R}^d \setminus \overline{D}$.
In particular, $G\star\chi_D$ is a function of class $C^{1,\gamma}$ in $D$ and in $\R^d\setminus\overline D$.
Hence the limits
 \begin{equation*}
\lim_{D \ni y \rightarrow x}\nabla G \star \chi_D (y) 
\qquad \text{and} \qquad \lim_{\overline{D} \not\ni y \rightarrow x}    \nabla G \star \chi_D (y)
\end{equation*} 
exist, for $x \in \partial D$, although they are not necessarily equal. 

We now show that they coincide tangentially. For that assume $d=2$ to simplify
the notation. Given $z_0=(x_0,y_0) \in \partial D$, take $r$ small enough so that, renaming the variables if necessary, there exists a function $\phi:I=(x_0-r, x_0+r)\to \R$ of class $C^{1,\gamma}$  such that
\begin{equation*}
Q(z_0,r) \cap D = \{(x,y)\in Q(z_0,r): \; y < \phi(x) \},
\end{equation*} 
where $Q(z_0,r)=  \{(x,y): |x-x_0|<r, \, |y-y_0|<r \}.$ 
Set $f(z)= G \star \chi_{D}(z), \; z \in \mathbb{\R}^2.$  Since $f$ is a Lipschitz function,
we have 
\begin{equation*}
f(x, \phi(x)-\varepsilon) \xrightarrow{\varepsilon \to 0} f(x, \phi(x)) \quad \text{uniformly in } I.
\end{equation*} 
Assume, without loss of generality, that $\tau((x,\phi(x)))=(1,\phi'(x))$, 
$x \in I.$ Thus we have that, in the weak convergence
of distributions on $I$,
\begin{equation*}
\langle \nabla f(x, \phi(x)-\varepsilon), \tau(x,\phi(x))\rangle = \frac{d}{dx} f(x, \phi(x)-\varepsilon) \xrightarrow{\varepsilon \to 0}
\frac{d}{dx} f(x, \phi(x)).
\end{equation*} 
Note that, in view of the H\"older regularity of $\nabla f$ in $D$, the convergence of the left-hand side is uniform in $I$. Since one can repeat the argument with $f(x, \phi(x)-\varepsilon)$ replaced by $f(x, \phi(x)+\varepsilon)$ the proof is complete.
\end{proof}

The next lemma establishes the behaviour of Calder\'on-Zygmund operators on ellipsoids.  
The behaviour on balls was first proved in \cite{MOV}, see also \cite{I87} for the special case of the Beurling transform.

\begin{lemma}\label{constel}
Let $T$ be an even homogeneous convolution Calder\'on-Zygmund operator in $\R^d$ of the form \eqref{CZ-i},
where the kernel $H$ is even, homogeneous of degree $-d$, integrable with respect to the $(d-1)$-dimensional surface measure on the unit sphere and satisfying the cancellation property
\eqref{canc}.
Let $E\subset \R^d$ be the domain enclosed by an ellipsoid. Then $T(\chi_E)$ is constant on $E.$ In particular, if
\begin{equation}\label{ell}
E=\left\{x=(x_1,...,x_d) \in \mathbb{R}^d: \frac{x_1^2}{a_1^2}+ \dots+\frac{x_d^2}{a_d^2} \le 1\right\},
\end{equation}
then the constant value of $T(\chi_E)$ in $E$ is

\begin{equation}\label{vconst}
-\frac{1}{2}\int_{|\xi|=1} \log \left(\frac{\xi_1^2}{a_1^2}+ \dots+\frac{\xi_d^2}{a_d^2}\right) H(\xi) \, d\sigma(\xi).
 \end{equation}
\end{lemma}

\begin{proof} 
Let $E\subset \R^d$ be the domain enclosed by an ellipsoid. With no loss of generality, up to a translation, we can assume that $E$ is centred at the origin. 
Then we can express $E$ as $E=Q(E_0),$ where $Q\in SO(d)$ is an appropriate rotation and $E_0$ is
the interior of an ellipsoid of the form \eqref{ell}.
Changing variables according to $Q$ we obtain
\begin{align*}
T(\chi_E)(z) &= \pv \int_E H(z-w) \,dw = \pv \int_{Q(E_0)} H(z-w) \,dw =  \pv \int_{E_0} H(z-Q(w)) \,dw\\
&=\pv \int_{E_0} (H\circ Q)(Q^{-1}z-w) \,dw = T_{H\circ Q}(\chi_{E_0})(Q^{-1}z),
\end{align*}
where $ T_{H\circ Q}$ is defined as in \eqref{CZ-i}, but with $H$ replaced by $H\circ Q$. Since $H\circ Q$ is of the same type as $
H,$ it is enough to prove the statement about the domain enclosed by an ellipsoid of the form  \eqref{ell}.

Take a point $z$ in the interior of $E$ and a radius  $R$ so big that $E \subset B(z,R).$ Then, by using the definition of the principal value, we have that 
\begin{equation}\label{aaa}
T(\chi_E)(z) =\left(\pv H \star \chi_E\right)(z) = \lim_{r \rightarrow 0} \int_{E \setminus B(z,r)} H(z-w)\,dw =
-\int_{B(z,R)\setminus E} H(z-w)\,dw.
\end{equation}
The last equality follows by writing 
$$
\int_{E \setminus B(z,r)} H(z-w)\,dw = \int_{B(z,R) \setminus B(z,r)} H(z-w)\,dw - \int_{B(z,R)\setminus E} H(z-w)\,dw,
$$ 
and by changing to polar coordinated centred at $z$ in the integral
\begin{align*}
\int_{B(z,R) \setminus B(z,r)} H(z-w)\,dw &= \int_r^R\int_{|\xi|=1} H(\rho \xi)\rho^{d-1}d\rho \,d\sigma(\xi)\\
&= \int_r^R\bigg(\int_{|\xi|=1} H( \xi) d\sigma(\xi)\bigg) \frac{1}{\rho}\, d\rho=0,
\end{align*}
where we have also used the homogeneity of $H$ and \eqref{canc}.

We now evaluate the last integral in \eqref{aaa} by taking again polar coordinates centred at $z$.  Given $\xi\in \R^d$ with $|\xi|=1$,  denote by $r(z,\xi)$ 
the unique positive number $r$ such that $z+r \xi$ lies in the boundary of the ellipsoid \eqref{ell}. Then
\begin{align*}
\int_{B(z,R)\setminus E} H(z-w)\,dw &
=  \int_{|\xi|=1}\int_{r(z,\xi)}^R H(t\xi) t^{d-1} \,dt d\sigma(\xi) \\
& =\int_{|\xi|=1}\int_{r(z,\xi)}^R H(\xi) \frac 1t \,dt d\sigma(\xi)= \int_{|\xi|=1} \log \left(\frac{R}{r(z,\xi)}\right) H(\xi)\, d\sigma(\xi),
\end{align*}
where we have used that $H$ is even and $(-d)$-homogeneous. Hence 
\begin{equation*}
\begin{split}
T(\chi_E)(z)  &= - \int_{|\xi|=1} \log \left(\frac{R}{r(z,\xi)}\right) H(\xi)\, d\sigma(\xi)
= - \int_{|\xi|=1} \log \left(\frac{1}{r(z,\xi)}\right) H(\xi)\, d\sigma(\xi) \\*[7pt]
& = -\frac{1}{2} \int_{|\xi|=1} \log \left(\frac{1}{r(z,\xi) r(z,-\xi)}\right) H(\xi)\, d\sigma(\xi),
\end{split}
\end{equation*}
where in the second identity we have used \eqref{canc} and in the third that $H$ is even.
There are exactly two points in the straight line $z+ r \xi, \; r \in \mathbb{R},$ that belong to the boundary of
\eqref{ell}. They correspond to the values $r=r(z,\xi)$ and $r=-r(z,-\xi).$ These values of the parameter $r$ 
are the solutions of
the second degree equation
\begin{equation*}
\left|\frac{\xi}{a}\right|^2 \,r^2+ 2  \Big\langle \frac{z}{a}, \frac{\xi}{a} \Big \rangle \, r + \left|\frac{z}{a}\right|^2-1 =0,
\end{equation*}
where we use the notation 
$$\frac{z}{a}= \left(\frac{z_1}{a_1}, \dots,\frac{z_d}{a_d}\right), \quad z=(z_1, \dots,z_d), \quad  a=(a_1, \dots,a_d),$$
and the brackets stand for scalar product in $\mathbb{R}^d.$
  Hence
\begin{equation*}
r(z,\xi) r(z,-\xi) = \left(1-\left|\frac{z}{a}\right|^2\right)\, \left|\frac{\xi}{a}\right|^{-2},
\end{equation*}
which yields \eqref{vconst} in view of \eqref{canc}.
\end{proof}

The following corollary is for the case $d=2$.

\begin{corollary}\label{coro}
Let $T$ be an even homogeneous Calder\'on-Zygmund singular integral in the plane, with kernel $H$ as in Lemma~\ref{constel}, and let $E$ be the compact set enclosed by a tilted ellipse, namely
 $$
 E= E(a,b,\varphi) = e^{i \varphi} \,\bigg\{(x,y) \in \mathbb{R}^2 : \frac{x^2}{a^2}+ \frac{y^2}{b^2} \le 1\bigg\}.
 $$
 Then the constant value of $T(\chi_E)$ on $E$ is
 \begin{equation}\label{constell}
 -\frac{1}{2}\int_{|\xi|=1} \log \bigg(\frac{{\langle \xi, e^{i \varphi} \rangle}^2}{a^2}+
 \frac{{\langle \xi, i e^{i \varphi} \rangle}^2}{b^2}\bigg) H(\xi) \, d\sigma(\xi). 
 \end{equation}
\end{corollary}

\begin{proof} Let $Q = e^{i\varphi}$ denote the rotation so that $E(a,b,\varphi) = Q(E_0)$, with $E_0=E(a,b,0)$. Proceeding as in the proof of Lemma~\ref{constel} we have that 
\begin{align*}
T(\chi_E)(z) &= T_{H\circ Q}(\chi_{E_0})(Q^{-1}z) = -\frac{1}{2}\int_{|\xi|=1} \log \Big(\frac{\xi_1^2}{a^2}+\frac{\xi_2^2}{b^2}\Big)(H\circ Q)(\xi) \, d\sigma(\xi)\\
&= -\frac{1}{2}\int_{|\xi|=1} \log \Big(\frac{\xi_1^2}{a^2}+\frac{\xi_2^2}{b^2}\Big) H(e^{i\varphi} \xi) \, d\sigma(\xi)\\
&= -\frac{1}{2}\int_{|\xi|=1} \log \bigg(\frac{{\langle \xi, e^{i \varphi} \rangle}^2}{a^2}+
 \frac{{\langle \xi, i e^{i \varphi} \rangle}^2}{b^2}\bigg) H(\xi) \, d\sigma(\xi).
\end{align*}
\end{proof}

\section{Existence and uniqueness of a compactly supported minimiser}
In this section we focus on the two-dimensional case $d=2$ and we show that the energy $I^\kappa$ in \eqref{en0} admits a unique minimiser for small perturbations $\kappa$.

\subsection{Existence of a minimiser} Existence of a minimiser for the energy functional \eqref{en0} follows from the direct method of the Calculus of Variations. Indeed, $I^\kappa$ is lower semicontinuous, since its overall kernel 
$$
W^\kappa(z-w) + \frac12(|z|^2+|w|^2)
$$
is lower semicontinuous and bounded from below (recall that $\kappa$ is bounded and continuous outside the origin, and $W^\kappa(0)=+\infty$). Moreover, 
\begin{align*}
I^\kappa(\mu) &= \int_{\R^2}\int_{\R^2}\left(W^\kappa(z-w) + \frac12(|z|^2+|w|^2)\right)\, d\mu(z) d\mu(w)\\
&\geq c\int_{\R^2}|z|^2 d\mu(z) -c',
\end{align*}
for some constants $c,c'>0$. Hence $I^\kappa$ is lower semicontinuous and satisfies the lower bound above, which guarantees tightness of minimising sequences. The same lower bound also guarantees that minimisers are compactly supported.

\subsection{Uniqueness of the minimiser}\label{FT}
We show that  $I^\kappa$ admits a unique minimiser, under the assumptions of Theorem~\ref{te}, by showing that the Fourier transform of $W^\kappa$ is strictly positive outside zero. 

Note that $W^\kappa = W^0 + \kappa$, where $W^0(z) = -\log|z|$, for $z \in \C$, $z\neq0$. 
We recall that 
\begin{equation}\label{recall}
\widehat{W^0}(\xi)= 2\pi \pv \frac{1}{|\xi|^2}, \quad \xi\in \C, \quad \xi \neq 0.
\end{equation}
Clearly $\widehat{W^0}(\xi)>0$ for $\xi \neq 0$. We now show that the assumptions on $\kappa$ ensure that $\hat \kappa$ has a similar structure as $\widehat{W^0}$, and that adding $\hat\kappa$ to $\widehat{W^0}$ does not disrupt its positivity.

To see this, let $\kappa$ be as in the statement of Theorem~\ref{te}. As a first step we compute the Fourier transform of $\kappa$. Consider the Fourier expansion of the restriction of $\kappa$ to the unit circle:
\begin{equation}\label{fou}
\kappa(e^{i \theta}) = \sum_{n=0}^\infty  a_n \cos(2n\theta) + b_n \sin(2n\theta), \quad \theta \in \R,
\end{equation}
where $a_n, b_n\in \R$. Note that only even frequencies appear in \eqref{fou} because $\kappa$ is an even function.
Hence, by zero-homogeneity,
\begin{equation}\label{foudins}
\kappa(z) = \sum_{n=0}^\infty  a_n \frac{\re{z^{2n}}}{|z|^{2n}} + b_n \frac{\operatorname{Im}{z^{2n}}}{|z|^{2n}}, \quad z \in \C, \quad z \neq 0.
\end{equation}
We have now rewritten $\kappa$ in a form that allows us to compute its Fourier transform. Indeed, since $\re{z^{2n}}$ and $\operatorname{Im}{z^{2n}}$ are homogeneous harmonic polynomials of degree $2n$, by using  \eqref{foupolk} with $d=2$, $k=2n$ and $s=2$, we obtain the Fourier transform identities 
\begin{equation*}
\frac{\re{z^{2n}}}{|z|^{2n}}  \xrightarrow{\rm{Fourier}} (-1)^{n} 4\pi n \pv \frac{\re{\xi^{2n}}}{|\xi|^{2n+2}},
\end{equation*}
\begin{equation*}
\frac{\operatorname{Im}{z^{2n}}}{|z|^{2n}}  \xrightarrow{\rm{Fourier}} (-1)^{n} 4\pi n \pv \frac{\operatorname{Im}{\xi^{2n}}}{|\xi|^{2n+2}},
\end{equation*}
where $\xi\neq 0$. 
Then the Fourier transform $\hat \kappa (\xi)$ of $\kappa$, for $\xi\neq 0$, is
\begin{equation*}\label{fouker}
\kappa(z)  \xrightarrow{\rm{Fourier}} \sum_{n=1}^\infty (-1)^{n} 4\pi n a_n \pv\frac{\re{\xi^{2n}}}{|\xi|^{2n+2}} + (-1)^{n} 4\pi n b_n \pv \frac{\operatorname{Im}{\xi^{2n}}}{|\xi|^{2n+2}}.
\end{equation*}
From the definition of $W^\kappa = W^0 + \kappa$ and \eqref{recall} we then have that, for $\xi \neq 0$,
$$
\widehat{W^\kappa}(\xi)=2\pi \pv \frac{1}{|\xi|^2}+\sum_{n=1}^\infty (-1)^{n} 4\pi n \left(a_n \pv\frac{\re{\xi^{2n}}}{|\xi|^{2n+2}} + b_n \pv \frac{\im{\xi^{2n}}}{|\xi|^{2n+2}}\right).
$$
Since we have that for some absolute constant $C_0>0$,
$$
\sum_{n=1}^\infty 2n(|a_n|+|b_n|) \leq 2 \bigg(\sum_{n=1}^\infty \frac{1}{n^2}\bigg)^{\frac12} 
\bigg(\sum_{n=1}^\infty n^4(|a_n|^2+|b_n|^2)\bigg)^{\frac12} \leq C_0 \left\| \frac{d^2 \kappa(e^{i\theta})}{d\theta^2} \right\|_{L^2(0,2\pi)},
$$
we can estimate, for $\xi\neq 0$,
\begin{equation*}\label{fouposit}
\begin{split}
\widehat{W^\kappa}(\xi)  & \ge \frac{2\pi}{|\xi|^2}\bigg(1- \sum_{n=0}^\infty 2n(|a_n|+|b_n|)\bigg)
\ge \frac{2\pi}{|\xi|^2}\Bigg(1- C_0\, \sup_{\substack{|z|=1\\ j\in\{1,2\}}} |\nabla^j k(z)|\Bigg) 
\ge \frac{\pi}{|\xi|^2},
\end{split}
\end{equation*}
provided
\begin{equation*}
C_0 \, \sup_{\substack{|z|=1\\ j\in\{1,2\}}} |\nabla^j \kappa(z)| \le \frac{1}{2},
\end{equation*}
which is true if $\ep_0$ in \eqref{small} is small enough.

Since the positivity of the Fourier transform $\widehat{W^\kappa}(\xi)$ for $\xi \neq 0$  yields strict convexity of the energy functional \eqref{en0} as a function of $\mu$ (see, e.g., \cite[Section~4]{MMRSVIzv}), we have that $I^\kappa$ has a unique minimiser.

\section{The Euler-Lagrange conditions} \label{sect:EL}
The minimiser $\mu$ of \eqref{en0} is characterised by two conditions, 
called the Euler-Lagrange conditions, which can be expressed in terms of a potential that we define as follows. The potential of $\mu \in \mathcal{P}(\R^2)$ is defined as 
\begin{equation}\label{pot}
P^\kappa(\mu)(z) = \left(W^\kappa\star \mu  \right) (z)+ \frac{1}{2}|z|^2, \quad z\in \mathbb{C},
\end{equation}
where $W^\kappa$ is the interaction potential \eqref{intpot}.

The first Euler-Lagrange condition is 
\begin{equation}\label{E1}
P^\kappa(\mu)(z) = C \quad \text{for }\mu\text{-a.e. } z \in \supp\mu,
\end{equation}
where $\supp\mu$ stands for the support of $\mu$, and $C$ is a constant.
The second Euler-Lagrange condition is
\begin{equation}\label{E2}
P^\kappa(\mu)(z) \ge C \quad \text{for } z \in \R^2 \setminus N \text{ with } \operatorname{Cap}(N)=0,
\end{equation}
where the constant is the same that appears in the first Euler-Lagrange condition. Here $\operatorname{Cap}$ is the logarithmic capacity.

We recall that for general energies conditions \eqref{E1}--\eqref{E2} are only necessary for minimality. For convex energies they are however also 
sufficient, and hence equivalent, to minimality. In our case, since $I^\kappa$ is strictly convex, the unique minimiser of $I^\kappa$ is the only measure satisfying \eqref{E1}--\eqref{E2}
for some constant $C$.
 
For the derivation of the Euler-Lagrange conditions see, e.g., \cite[Theorem~1.3]{ST}.

\subsection{The first Euler-Lagrange condition \eqref{E1}}\label{sect:EL11}
We show in this section that under the assumptions of Theorem~\ref{te} there exists an ellipse with interior $E=E(a,b,\varphi)$, defined as in \eqref{ellrot}, such that the potential $P^\kappa(\mu)$ of the probability measure
$\mu= \chi_E/|E|$ (see \eqref{pot}) is constant on $E$.

\subsubsection{Computing the Hessian of $P^\kappa$ on $E$}\label{subsub:H} We note that condition \eqref{E1} for $\mu=\chi_E / |E|$ is equivalent of the vanishing on $E$ of the Hessian $H(P^\kappa(\mu))$ of $P^\kappa(\mu)$.

To see this, we first observe that the potential $P^\kappa$ of $\chi_E / |E|$  is of class $C^1$ with Lipschitz continuous gradient. Hence, finding an ellipse $E$ such that $\mu=\chi_E / |E|$ satisfies \eqref{E1} is equivalent to finding $E$ such that $\nabla P^\kappa(\mu)=0$ on $E$. Since $P^\kappa(\mu)$ is an  even function,
$\nabla P^\kappa(\mu)$ is odd and thus it vanishes at the origin. Consequently $\nabla P^\kappa(\mu)=0$ on $E$ is equivalent to the vanishing in $E$ of
the Hessian $H(P^\kappa(\mu))$ of $P^\kappa(\mu)$.  

We now compute the Hessian of $P^\kappa(\mu)$ in $E$. To this aim we will 
use Lemma~\ref{constel}, in dimension $d=2$, with kernel $H=\partial_{ij} W^\kappa$ (in fact separately with $H=\partial_{ij} W^0$ and $H=\partial_{ij} \kappa$).

As a first step, we compute the second derivatives of the terms in $W^\kappa$. 
Since the kernel $\kappa$ is homogeneous of degree $0$ and of class $C^3$  off the origin, each second order derivative $\partial_{ij} \kappa$ of $\kappa$
in the  sense of distributions is a constant multiple of the Dirac delta at the origin plus the principal value distribution 
associated with the kernel $\partial_{ij} \kappa$. More precisely
\begin{equation}\label{sod1}
\partial_{11} \kappa = \left( \int_{|\xi|=1} \xi_1 \,\partial_1  \kappa(\xi) \, d\sigma(\xi)\right) \delta_0 + \pv\partial_{11} \kappa,
\end{equation}
\begin{equation}\label{sod2}
\partial_{22} \kappa = \left( \int_{|\xi|=1} \xi_2 \,\partial_2  \kappa(\xi) \, d\sigma(\xi)\right) \delta_0 + \pv\partial_{22} \kappa,
\end{equation}
\begin{equation}\label{sod3}
\partial_{12} \kappa = \left( \int_{|\xi|=1} \xi_1 \,\partial_2  \kappa(\xi) \, d\sigma(\xi)\right) \delta_0 + \pv\partial_{12} \kappa.
\end{equation}
These formulas follow from checking the action of the left-hand sides on a test function and applying integration by parts via
Green-Stokes.

Arguing as in the proof of \eqref{Gi0} in Lemma~\ref{tangcont}, one can show that 
\begin{equation}\label{kij0}
\int_{|\xi|=1} \partial_{ij} \kappa(\xi) \, d\sigma(\xi) =0, \quad i,j=1, 2. 
\end{equation}
Since $\kappa$ is even, $\partial_{ij}\kappa$ is the kernel of an even homogeneous Calder\'on-Zygmund 
operator to which one can apply Lemma~\ref{constel}. Therefore the distributional second order derivatives of $\kappa \star \chi_E$ 
are constant on $E$. This shows the relevance of the assumption that $\kappa$ is even in Theorem~\ref{te}.

Identities \eqref{sod1}, \eqref{sod2}, and \eqref{sod3} apply also to the kernel $-\log|z|$, which, in fact, may be thought of 
as being homogeneous of degree $0$. In this case the constant multiple of $\delta_0$ can be computed explicitly, since 
$$
\int_{|\xi|=1} \xi_1 \,\partial_1(\log|\xi|) \, d\sigma(\xi)=\int_{|\xi|=1} \frac{\xi_1^2}{|\xi|^2} \, d\sigma(\xi) = \int_0^{2\pi} \cos^2\theta \,d\theta = \pi,
$$
and similarly
$$
\int_{|\xi|=1} \xi_2 \,\partial_2(\log|\xi|) \, d\sigma(\xi)= \pi, \qquad \int_{|\xi|=1} \xi_1 \,\partial_2(\log|\xi|) \, d\sigma(\xi)=0.
$$
Hence we obtain 
\begin{equation}\label{sod1log}
\partial_{11}\left(-\frac{1}{\pi}\log|z|\right) = -\delta_0 + \frac{1}{\pi}\, \pv \frac{x^2-y^2}{|z|^4},
\end{equation}
\begin{equation}\label{sod2log}
\partial_{22}\left(-\frac{1}{\pi}\log|z|\right) = -\delta_0 - \frac{1}{\pi}\, \pv \frac{x^2-y^2}{|z|^4},
\end{equation}
\begin{equation}\label{sod3log}
\hspace{-1.9cm}\partial_{12}\left(-\frac{1}{\pi}\log|z|\right) =  \frac{1}{\pi}\, \pv \frac{2xy}{|z|^4}.
\end{equation}
Moreover, \eqref{kij0} still holds true when $\kappa$ is replaced by $-\log|\cdot|$. Therefore the distributional second order derivatives of $-\log|\cdot| \star \chi_E$ are constant on $E$.

Since $P^\kappa(\frac{\chi_E}{|E|}) = -\log|\cdot| \star \frac{\chi_E}{|E|}+ \kappa\star \frac{\chi_E}{|E|}+\frac12|\cdot|^2$, we have then proved that every second order derivative of $P^\kappa(\frac{\chi_E}{|E|})$ is constant 
in $E$.

\subsubsection{Imposing that the Hessian of $P^\kappa$ is zero in $E$}\label{subsub:sys} The first Euler-Lagrange condition \eqref{E1} is equivalent to the vanishing of the Hessian of $P^\kappa$ in $E$. Since the Hessian is a symmetric matrix, 
requiring that it vanishes on $E$ yields a system of three equations in the parameters $a,b$ and $\varphi$. In this section we write this system explicitly. In Section~\ref{subsub:IFT} we will show that this system is uniquely solvable under the hypotheses of Theorem~\ref{te}, and hence that there exists a unique ellipse $E$ with semi-axes $a$ and $b$, rotated of an angle $\varphi$ with respect to the $x$-axis, such that the associated $P^\kappa(\frac{\chi_E}{|E|})$ satisfies \eqref{E1}.

The vanishing of the Hessian of $P^\kappa(\frac{\chi_E}{|E|})$ in $E$ corresponds to the three equations
\begin{equation}\label{3eq}
 \partial_{11}P^\kappa\left(\frac{\chi_E}{|E|}\right) = 0, \quad \partial_{22}P^\kappa\left(\frac{\chi_E}{|E|}\right) = 0, \quad \text{and} \quad \partial_{12}P^\kappa\left(\frac{\chi_E}{|E|}\right) = 0 \quad \text{in } E.
\end{equation}
So far, in Section~\ref{subsub:H} we have shown that every second order derivative of $P^\kappa(\frac{\chi_E}{|E|})$ is constant 
in $E$, and so we now need to show that the constant value of every second order derivative of $P^\kappa$ in $E$ is in fact zero. 

To make the three equations \eqref{3eq} as explicit as possible we need to know the constant value on $E=E(a,b,\varphi)$ 
of $\pv \frac{x^2-y^2}{|z|^4}\star
\chi_E$ and of $\pv \frac{2xy}{|z|^4}\star
\chi_E$, from \eqref{sod1log}--\eqref{sod3log}. One could resort to \eqref{constell}, but, setting $E_0=E(a,b,0),$ it is faster to appeal to the well-known formula 
(see \cite[page 1408]{HMV})
\begin{equation}\label{f:HMV}
 \left(\frac{1}{\pi z}\star \chi_{E_0}\right)(z)= \bar{z}-\lambda z \quad \text{for }z \in E_0, \quad \lambda=\frac{a-b}{a+b}.
\end{equation}
Changing variables to pass from $E$ to $E_0$, and denoting $E=Q(E_0)$, with $Q=e^{i\varphi}$, we have that, by using \eqref{f:HMV},
\begin{align*}
\left(\frac{1}{\pi z}\star \chi_{E}\right)(z) &= \frac{1}{\pi}\int_{E_0}\frac{1}{z-e^{i\varphi}w} \,dw =  e^{-i\varphi}\left(\frac{1}{\pi}\int_{E_0}\frac{1}{e^{-i\varphi}z-w} \,dw\right)\\
&= e^{-i\varphi}\big(\overline{e^{-i\varphi}z}-\lambda e^{-i\varphi}z\big) = e^{-i\varphi}\big(e^{i\varphi}\bar{z}-\lambda e^{-i\varphi}z\big)\\
&=\bar z - \lambda e^{-2i\varphi}z \quad \text{for }z\in E.
\end{align*}
Differentiating in $z$ we obtain 
\begin{equation*}
\pv\left(\frac{1}{\pi z^2}\star \chi_{E}\right)(z)= \lambda e^{-i2\varphi}, \quad z \in E,
\end{equation*}
and taking real parts and imaginary parts we get respectively
\begin{equation}\label{realBeurlingE2}
\pv \bigg(\frac{1}{\pi} \frac{x^2-y^2}{|z|^4}\star \chi_{E}\bigg)(z)= \lambda \cos(2\varphi), \quad z \in E
\end{equation}
and
\begin{equation}\label{imBeurlingE2}
\pv \left(\frac{1}{\pi} \frac{2xy}{|z|^4}\star \chi_{E}\right)(z)= \lambda \sin(2\varphi), \quad z \in E.
\end{equation}

We are now ready to rewrite the system \eqref{3eq} more explicitly. For convenience, we use the variables $(p,\lambda, \varphi)$, with $p=ab$, instead of $(a,b,\varphi)$. (Alternatively, one could consider $(p,q,\varphi)$, with $q=\frac{a}{b}$, observing that $q$ can be obtained from $\lambda$ as $q=\frac{1+\lambda}{1-\lambda}$.)

We start with the first equation in \eqref{3eq}, namely $\partial_{11} P^\kappa(\frac{\chi_E}{|E|})=0$ on $E$. By \eqref{sod1log}, \eqref{realBeurlingE2}, \eqref{sod1} (and multiplying the equation by $p$) we have that, for $z\in E$, 
\begin{equation}\label{eq0}
p -1 +\cos(2\varphi) \lambda + \frac1\pi \partial_{11} \kappa\star \chi_E(z) + I_1(\kappa)=0,
\end{equation}
where 
$$
I_1(\kappa)=  \frac{1}{\pi} \int_{|\xi|=1} \xi_1 \,\partial_1  \kappa(\xi) \, d\sigma(\xi).
$$
We compute the convolution term in \eqref{eq0} by using Lemma~\ref{constel} and Corollary~\ref{coro}, since $\partial_{11} \kappa$ is admissible as kernel $H$ (also thanks to \eqref{kij0}): for $z\in E$,
\begin{align*}
\frac1\pi \partial_{11} \kappa\star \chi_E(z)&=  -\frac{1}{2\pi}\int_{|\xi|=1} \log \bigg(\frac{{\langle \xi, e^{i \varphi} \rangle}^2}{a^2}+
 \frac{{\langle \xi, i e^{i \varphi} \rangle}^2}{b^2}\bigg) \partial_{11} \kappa(\xi) \, d\sigma(\xi)\\
 &= -\frac{1}{2\pi}\int_{|\xi|=1}\bigg(\log\left(\frac{1}{a^2}\right) +  \log \left({\langle \xi, e^{i \varphi} \rangle}^2+q^2
{\langle \xi, i e^{i \varphi} \rangle}^2\right)\bigg) \partial_{11} \kappa(\xi) \, d\sigma(\xi)\\
&=-\frac{1}{2\pi}\int_{|\xi|=1}  \log \left({\langle \xi, e^{i \varphi} \rangle}^2+q^2
{\langle \xi, i e^{i \varphi} \rangle}^2\right) \partial_{11} \kappa(\xi) \, d\sigma(\xi).
\end{align*}
Hence \eqref{eq0} becomes 
\begin{equation}\label{eq1}
p -1 +\cos(2\varphi) \lambda + F_1(\lambda, \varphi,\kappa) + I_1(\kappa)=0,
\end{equation}
where
\begin{equation}\label{efe1}
F_1(\lambda,\varphi,\kappa) = -\frac{1}{2\pi} \int_{|\xi|=1} \log\left({\langle \xi, e^{i\varphi} \rangle}^2+ 
q^2  {\langle \xi,i e^{i\varphi} \rangle}^2 \right)\,\partial_{11} \kappa(\xi)\,d\sigma(\xi).
\end{equation}
By the same token the second equation in \eqref{3eq} times $p$ is
\begin{equation}\label{eq2}
p -1 -\cos(2 \varphi) \lambda  +F_2(\lambda, \varphi,\kappa)+ I_2(\kappa)=0,
\end{equation}
where
\begin{equation}\label{efe2}
F_2(\lambda,\varphi,\kappa) = -\frac{1}{2\pi} \int_{|\xi|=1} \log\left({\langle \xi, e^{i\varphi} \rangle}^2+ 
q^2  {\langle \xi,i e^{i\varphi} \rangle}^2 \right)\,\partial_{22} \kappa(\xi)\,d\sigma(\xi),
\end{equation}
and
$$
I_2(\kappa)= \frac{1}{\pi}  \int_{|\xi|=1} \xi_2 \,\partial_2  \kappa(\xi) \, d\sigma(\xi).
$$
Finally, the third equation in \eqref{3eq} times $p$ is
\begin{equation}\label{eq3}
 \sin(2\varphi) \lambda + F_3(\lambda,\varphi,\kappa)+ I_3(\kappa)=0,
\end{equation}
where
\begin{equation}\label{efe3}
F_3(\lambda,\varphi,\kappa) = -\frac{1}{2\pi} \int_{|\xi|=1} \log\left({\langle \xi, e^{i\varphi} \rangle}^2+ 
q^2  {\langle \xi,i e^{i\varphi} \rangle}^2 \right)\,\partial_{12} \kappa(\xi)\,d\sigma(\xi),
\end{equation}
and
$$
I_3(\kappa)= \frac{1}{\pi}  \int_{|\xi|=1} \xi_1 \,\partial_2  \kappa(\xi) \, d\sigma(\xi).
$$
In conclusion, the system \eqref{3eq} is equivalent to the three equations \eqref{eq1}, \eqref{eq2}, and \eqref{eq3}, in the three unknowns $(p,\lambda,\varphi)$, namely 
\begin{equation}\label{systemaaa}
\begin{cases}
p -1 +\cos(2\varphi) \lambda+ F_1(\lambda, \varphi,\kappa) + I_1(\kappa)=0,\\
p -1 -\cos(2 \varphi) \lambda  +F_2(\lambda, \varphi,\kappa)+ I_2(\kappa)=0,\\
\sin(2\varphi) \lambda+ F_3(\lambda,\varphi,\kappa) + I_3(\kappa)=0.
\end{cases}
\end{equation}

\subsubsection{Solving \eqref{systemaaa} via the Implicit Function Theorem}\label{subsub:IFT}

We want to show that the system \eqref{systemaaa} admits a (unique) solution, at least for $\kappa$ small enough in $C^3$ norm (see assumption~\eqref{small} in Theorem~\ref{te}). The idea is to use the Implicit Function Theorem to find, for $\kappa$ small, a solution of  \eqref{systemaaa} `close' to the solution for $\kappa=0$. 

To explain our strategy let us first consider the system
\begin{equation}\label{newsystem}
\mathcal L(p,\lambda, \varphi) : = (p -1 +\cos(2\varphi) \lambda, \,p -1 -\cos(2\varphi) \lambda,\, \sin(2\varphi) \lambda)=0.
\end{equation}
Note that \eqref{newsystem} is the `limiting' system for \eqref{systemaaa}. Indeed the quantities $I_j$ and $F_j$, $j=1,2,3$, in \eqref{systemaaa} are small owing to the smallness assumption on $\kappa$. 

Since the minimiser of $I^0$ in \eqref{en0} (corresponding to $\kappa=0$) is the normalised characteristic function of the unit disc, 
one would like to examine  \eqref{systemaaa} and \eqref{newsystem}  for $(p, \lambda)$ close to $(1,0).$ We have $\mathcal L(1,0,\varphi)=0$ for each 
angle $\varphi,$ which is consistent with the fact that the support of the minimiser is a disc. Unfortunately, the fact that the last column of the gradient of $\mathcal L$ at $(1,0,\varphi)$ vanishes identically prevents us from applying directly the Inverse Function Theorem. To overcome this difficulty we need to examine more carefully the three equations in \eqref{systemaaa}. This careful analysis will in particular identify the rotation angle of the ellipse solution of \eqref{systemaaa}, which is by now undetermined due to the isotropy of the disc.

As a first step, in the next lemma we compute the integral $I_3(\kappa)$ in terms of the Fourier coefficients of $\kappa$. Indeed, we recall that in Section~\ref{FT}, in view of the Fourier series expansion of $\kappa(e^{i\theta}),$ we concluded that
\begin{equation}\label{foudins2}
\kappa(z) = \sum_{n=0}^\infty  a_n \frac{\re{z^{2n}}}{|z|^{2n}} + b_n \frac{\im{z^{2n}}}{|z|^{2n}}, \quad z \in \C, \quad z \neq 0.
\end{equation}

\begin{lemma}\label{bone}
We have
$$
I_3(\kappa)= \frac{1}{\pi} \int_{|\xi|=1} \xi_1 \, \partial_2  \kappa(\xi) \, d\sigma(\xi) =  b_1.
$$
\end{lemma}
\begin{proof}
It is more convenient to perform the calculation in the complex variables $z$ and $\bar{z}.$ We have
 \begin{equation*}
  I_3 = \im \frac{1}{\pi}  \int_{|z|=1} z\, \overline{\partial} \kappa(z) \, d\sigma(z)
  + \im \frac{1}{\pi}  \int_{|z|=1} \bar{z}\, \overline{\partial} \kappa (z) \, d\sigma(z).
 \end{equation*}
To compute $\overline{\partial} \kappa$ set
\begin{equation*}
\frac{\re{z^{2n}}}{|z|^{2n}} = \frac{1}{2}\left(\frac{z^n}{\bar{z}^n}+\frac{\bar{z}^n}{z^n}\right) \quad \quad \text{and} \quad \quad  \frac{\im{z^{2n}}}{|z|^{2n}}= \frac{1}{2i}\left(\frac{z^n}{\bar{z}^n}-\frac{\bar{z}^n}{z^n}\right) 
\end{equation*} 
and then take $\overline{\partial}$ to obtain
\begin{equation}\label{dbark}
\overline{\partial} \kappa(z) = \frac12  \sum_{n=1}^\infty  \left(n a_n \left(-\frac{z^n}{\bar{z}^{n+1}}+\frac{\bar{z}^{n-1}}{z^{n}}\right) - \frac{n b_n}{i} \left(\frac{z^n}{\bar{z}^{n+1}}+\frac{\bar{z}^{n-1}}{z^{n}}\right)\right), \quad z \in \C, \quad z \neq 0.
\end{equation}
Hence we have that, for $z\neq 0$,
\begin{align*}
z\overline{\partial} \kappa(z) &= \frac12  \sum_{n=1}^\infty  \left(n a_n \left(-\frac{z^{2n+2}}{|z|^{2n+2}}+\frac{\bar{z}^{2n-2}}{|z|^{2n-2}}\right) - \frac{n b_n}{i} \left(\frac{z^{2n+2}}{|z|^{2n+2}}+\frac{\bar{z}^{2n-2}}{|z|^{2n-2}}\right)\right),\\
\bar z\overline{\partial} \kappa(z) &= \frac12  \sum_{n=1}^\infty  \left(n a_n \left(-\frac{z^{2n}}{|z|^{2n}}+\frac{\bar{z}^{2n}}{|z|^{2n}}\right) - \frac{n b_n}{i} \left(\frac{z^{2n}}{|z|^{2n}}+\frac{\bar{z}^{2n}}{|z|^{2n}}\right)\right).
\end{align*}
By integrating $z\overline{\partial} \kappa(z)$ and $\bar z\overline{\partial} \kappa(z)$ on the unit circle one can easily see that all the frequencies $n\neq 1$ yield a zero integral, and that 
\begin{equation*}
\frac{1}{\pi}  \int_{|z|=1}  z \, \overline{\partial} \kappa(z)\,d\sigma(z) =  a_1+ib_1\quad\quad \text{and}\quad\quad 
\frac{1}{\pi} \int_{|z|=1}  \bar{z} \, \overline{\partial} \kappa(z)\,d\sigma(z) = 0.
\end{equation*} 
\end{proof}

In view of Lemma~\ref{bone}, system \eqref{systemaaa} becomes 
\begin{equation}\label{systemaab}
\begin{cases}
p -1 +\cos(2\varphi) \lambda+ F_1(\lambda, \varphi,\kappa) + I_1(\kappa)=0,\\
p -1 -\cos(2 \varphi) \lambda  +F_2(\lambda, \varphi,\kappa) + I_2(\kappa)=0,\\
\sin(2\varphi) \lambda+ F_3(\lambda,\varphi,\kappa) + b_1=0.
\end{cases}
\end{equation} 
We now show that, up to a rotation of the axes, we can always assume that $b_1=0$. Consider the change of coordinates $w=(u,v)=z e^{i \psi}$, with the angle $\psi$ to be fixed later (see \eqref{angle}). In these new coordinates, from \eqref{foudins2} we have that 
\begin{align*}
\tilde \kappa(w):= \kappa(z(w)) = \sum_{n=0}^\infty A_n\frac{\re{w^{2n}}}{|w|^{2n}}
+ B_n\frac{\im{w^{2n}}}{|w|^{2n}},
\end{align*}
where
$$
A_n = a_n \cos(2n\psi)-b_n\sin(2n\psi), \quad B_n = a_n \sin(2n\psi)+b_n\cos(2n\psi).
$$
On the other hand, $|z|=|w|$, and hence the logarithmic and the confinement terms in the potential $P^\kappa(\chi_E/|E|)$ are unchanged. 
By choosing the rotation angle $\psi$ so that
\begin{equation}\label{angle}
a_1\sin(2\psi)+b_1\cos(2\psi)=0
\end{equation} 
we get $B_1=0,$ which means that in the rotated variables $w=(u,v)$ there will be no term $I_3(\kappa)$ in the third equation in \eqref{systemaaa} (or \eqref{systemaab}).

This is not yet the angle to which the statement of Theorem~\ref{te} refers. Once we find the angle $\tilde\varphi$ that the candidate ellipse in the $(u,v)$ coordinate system forms with the $u$-axis, the angle $\varphi$ of Theorem~\ref{te} is obtained as $\varphi=\tilde\varphi-\psi$ (see also Remark~\ref{rem:psi}).  

We then assume from now on that $b_1=0$, and so the third equation in \eqref{systemaab} is
$$
\sin(2\varphi) \lambda + F_3(\lambda,\varphi,\kappa) = 0,
$$
with $F_3$  given by \eqref{efe3}. 

The system \eqref{systemaab}, with $b_1=0$, can then be written as 
\begin{equation}\label{sys:L}
L(p,\lambda,\varphi,\kappa)=0, 
\end{equation}
where the components of $L=(L_1,L_2,L_3)$ are given by 
\begin{equation*}
L_1(p,\lambda, \varphi, \kappa):= p -1 +\cos(2\varphi) \lambda + F_1(\lambda, \varphi, \kappa) + I_1(\kappa),
\end{equation*} 
\begin{equation*}
L_2(p,\lambda, \varphi, \kappa):= p -1 -\cos(2\varphi) \lambda + F_2(\lambda, \varphi, \kappa) + I_2(\kappa),
\end{equation*} 
and
\begin{equation*}
L_3(p,\lambda, \varphi, \kappa):= \sin(2\varphi)\lambda +  F_3(\lambda,\varphi, \kappa).
\end{equation*} 
Note that $L(1,0,0,0)=0$. Hence, for $\kappa$ small we look for a solution $(p,\lambda,\varphi)$ of the system \eqref{sys:L} close to $(1,0,0)$. 

Unfortunately, also the  system \eqref{sys:L} is not suitable for the application of the Inverse Function Theorem, since $\frac{\partial L}{\partial \varphi}(1,0,0,0)=0$. Indeed, 
$F_j(0, \varphi,\kappa)=0$ for every $\varphi \in \mathbb{R}$ and $\kappa\in C^2(\mathbb T)$, $j=1,2,3$. This follows from the fact that $\lambda =0$ means that $a=b$ and so $q=1$ and, for $|\xi|=1$
$$
\log\left({\langle \xi, e^{i\varphi} \rangle}^2+ 
q^2  {\langle \xi,i e^{i\varphi} \rangle}^2 \right) = \log |\xi|^2=0.
$$ 
Consequently,
\begin{equation}\label{a}
\frac{\partial F_j}{\partial \varphi}(0,\varphi,\kappa)=0, \quad j=1,2,3.
\end{equation} 

We then modify the  system \eqref{sys:L} slightly by dividing the third equation ($L_3=0$) by $\lambda$. More precisely, we consider the system given by
\begin{equation}\label{sys}
G(p,\lambda, \varphi, \kappa)=0
\end{equation} 
where $G=(G_1,G_2,G_3)$ is given in components by
\begin{equation*}
G_1(p,\lambda, \varphi,  \kappa):= L_1(p,\lambda, \varphi, \kappa),
\qquad
G_2(p,\lambda, \varphi,  \kappa):= L_2(p,\lambda, \varphi, \kappa),
\end{equation*} 
and
\begin{equation*}
G_3(p,\lambda, \varphi,  \kappa):=  
\begin{cases} 
\medskip
\sin(2\varphi) + \dfrac{F_3(\lambda,\varphi,  \kappa)}{\lambda} &\quad \textrm{if } \lambda\neq 0,\\
\sin(2\varphi) + \dfrac{\partial F_3}{\partial\lambda} (0,\varphi,  \kappa) &\quad \textrm{if } \lambda= 0.
\end{cases}
\end{equation*} 
We claim that 
\begin{equation*}
G(1,0, 0, 0) =0.
\end{equation*} 
Since we have already proved that $L(1,0,0)=0$, and the two systems only differ in their third component, it only remains to prove that $G_3(1,0,0,0)=0$. This can be readily seen using the fact that $F_3(\lambda,\varphi,0)=0$ for every $\lambda\in(-1,1)$, $\varphi\in\R$.

On the other hand the Jacobian matrix of $G$ with respect to the variables $p, \lambda$ and $\varphi$ at $(1,0,0,0)$ is
\begin{equation}\label{jac0}
 \frac{ \partial G}{\partial (p,  \lambda,  \varphi)}(1,0,0,0)=
\left( \begin{matrix}
 1 & 1 &0 \\
 1& -1 & 0 \\
 0 & 0 & 2
 \end{matrix}
 \right),
\end{equation} 
hence it is invertible. By the Implicit Function Theorem in Banach spaces (see, e.g., \cite[statement (10.2.1)]{D}) there exist $p(\kappa)$, $\lambda(\kappa)$, $\varphi(\kappa)$
satisfying the system \eqref{sys} for $\kappa$ close to zero in the $C^2$-norm on $|\xi|=1$.  Note that the functions $G$ and $ \frac{\partial G}{\partial (p,  \lambda,  \varphi)}$ are clearly continuous in all the variables $(p,\lambda,\varphi,\kappa)$, where $(p,\lambda,\varphi)\in (0,+\infty)\times (-1,1)\times \R$, and $\kappa$ belongs to the Banach space of $C^2$-functions on the unit sphere equipped with the $C^2$-norm. This is in fact sufficient for our conclusion, and we need not prove that $G$ is continuously differentiable in all the variables.

As observed above, $p$ and $\lambda$ determine the semi-axes $a$ and $b$ of the ellipse. The angle $\varphi$ we have obtained here is in fact the angle $\tilde \varphi$ of the rotation of the ellipse with respect to the coordinate frame $w=(u,v)$. Note that $a$ and $b$ are close to $1$ and $\tilde \varphi$ is close to $0$. Coming back from the $(u,v)$-plane to the original frame our ellipse has semi-axes close to $1$ and a clockwise rotation angle $\varphi = \tilde{\varphi}-\psi$, hence close to the angle $-\psi$ defined in \eqref{angle} (see also Remark~\ref{rem:psi}). Alternatively, looking from the perspective of the original $(x,y)$-plane, we have rotated the ellipse of an angle $\psi$ counterclockwise and then of an angle $\tilde \varphi$ clockwise. 
In conclusion, we have then found an ellipse such that the potential of the normalised characteristic function of its interior satisfies the first Euler-Lagrange equation \eqref{E1}.

\begin{remark}[Special case: $\kappa$ even in each variable]\label{rem:sep-even} 
The proof becomes shorter if the kernel $\kappa$ is even in each variable separately. If this is the case, $\partial_2 \kappa$ is odd in $y$ and so $\partial_{12}\kappa$ is odd in $y$ too. Thus 
$\pv \int_E \partial_{12}\kappa (z) d z=0$ for the interior $E$ of each ellipse centred at the origin with axes on the 
coordinate axes, 
by Fubini (fixing $x$ and integrating in $y$). Then the constant value of $\pv \partial_{12}\kappa \star \chi_E$ on $E$ is $0,$ provided we look only at ellipses with $\varphi=0$. Moreover, the factor of $\delta_0$ in \eqref{sod3} is also zero.  Then the third equation in \eqref{3eq} reduces to 
$$
\partial_{12}\left(-\log|z|\right)\star \frac{\chi_{E}}{|E|} =0  \quad \text{on } E,
$$
which by \eqref{sod3log} and \eqref{imBeurlingE2} is satisfied, since $\varphi=0$.

Therefore the system \eqref{3eq} is equivalent to the two conditions 
$$\partial_{11}P^\kappa(\chi_E/|E|)=0, \qquad \partial_{22}P^\kappa(\chi_E/|E|)=0$$ 
in the two unknowns $(p,\lambda)$ (and $\varphi=0$), namely to the first two equations in the system \eqref{systemaab}, with $\varphi=0$:
\begin{equation*}
\begin{cases}
p -1 +  \lambda + F_1(\lambda, 0, \kappa)+ I_1(\kappa)=0, \\
p -1 -\lambda + F_2(\lambda, 0,\kappa)+I_2(\kappa)=0,
\end{cases}
\end{equation*} 
We can immediately apply the Implicit Function Theorem to the system above, which gives $p$ and $\lambda$ in terms of $\kappa$. 
\end{remark}

\begin{remark}[The angle $-\psi$]\label{rem:psi}
The angle $-\psi$ is the rotation angle with respect to the $x$-axis of the minimising ellipse corresponding to the kernel
\begin{equation}\label{psi-per}
-\log|z|+a_0+b_0+a_1 \frac{\re{z^{2}}}{|z|^{2}} + b_1 \frac{\operatorname{Im}{z^{2}}}{|z|^{2}}, \quad z \in \C, \quad z \neq 0,
\end{equation}
for $a_1,b_1$ small enough. In \eqref{psi-per} the perturbation is given by keeping only the first two terms in the Fourier expansion \eqref{foudins} of $\kappa$. 

Indeed, by \cite[Section~4.2]{CMMRSV}, if $a_1^2+b_1^2<1/4$ and $b_1\neq0$, the minimiser for the kernel \eqref{psi-per} is the normalised characteristic function of an ellipse, whose major axis is rotated with respect to the $x$-axis of an angle $\theta$ satisfying
$$
\tan\theta=-\frac{a_1+\sqrt{a_1^2+b_1^2}}{b_1}.
$$
It is immediate to check that the solutions of \eqref{angle} for $b_1\neq0$ satisfy
$$
\tan\psi=\frac{a_1\pm\sqrt{a_1^2+b_1^2}}{b_1},
$$
hence, up to integer multiples of $\pi$, either $\psi=-\theta$ or $\psi=-\theta+\pi/2$ (the rotation angle of the minor axis of the ellipse).
\end{remark}

\subsection{The second Euler-Lagrange condition \eqref{E2}}\label{sect:EL22} Let $E^\kappa=E(a^\kappa,b^\kappa,\varphi^\kappa)$ denote the interior of the ellipse of the type \eqref{ellrot} obtained as the solution of the first Euler-Lagrange condition in Section~\ref{sect:EL11}. We recall that for $\kappa=0$ the unique minimiser of $I^0$ is the normalised characteristic function of the unit disc. Hence we have that $E^0$ is the closed unit ball $B(0,1)$.

In this section we prove that $E^\kappa$ also satisfies the second Euler-Lagrange condition~\eqref{E2}. 
This will conclude the characterisation of the unique minimiser of the energy $I^\kappa$ in \eqref{en0}, 
with $\kappa$ as in the statement of Theorem~\ref{te}, as the normalised characteristic function of~$E^\kappa$.

The strategy of proof is quite simple. First we note that for $\kappa=0$ the Euler-Lagrange condition \eqref{E2} for $P^0$ is satisfied with a strict inequality outside $B(0,1)$. Then, since for small $\kappa$ we have that $E^\kappa$ is close to $B(0,1)$, we deduce that condition \eqref{E2} for $P^\kappa$ is satisfied, for $\kappa\neq 0$ small enough, in a neighbourhood of $B(0,1)$ -- a security region. We can in fact prove that the security region is uniform in $\kappa$ under the assumption that  the smallness of $\kappa$ is controlled in the $C^3$-norm. Finally, this shows \eqref{E2}.

\subsubsection{Subharmonicity of the potentials}\label{sect:subh}
For brevity we denote with $P^\kappa$ the potential of $\chi_{E^\kappa}/|E^\kappa|$ defined as in \eqref{pot} (hence omitting the argument $\chi_{E^\kappa}/|E^\kappa|$), namely 
\begin{equation}\label{potell2}
P^\kappa(z)= \left(\left(-\log|\cdot| +\kappa\right) \star \frac{\chi_{E^\kappa}}{|E^\kappa|} \right) (z)+ \frac{1}{2}|z|^2, \quad z
\in \mathbb{\C}.
\end{equation}

\begin{lemma}\label{lapposk} 
For every $z \in \partial E^\kappa$ the limit 
of $\Delta P^\kappa(w)$ as $w\to z$, for $w\notin E^\kappa$, exists and satisfies the lower bound 
\begin{equation}\label{oldjump}
\lim_{ E^\kappa \not \ni w \to z} \Delta P^\kappa(w) \ge  1, 
\end{equation}
provided the number $\varepsilon_0$ in the statement of Theorem~\ref{te} is small enough.
\end{lemma}

\begin{proof}
For $w\not\in E^\kappa$ we have that 
\begin{equation}\label{lappka}
\Delta P^\kappa (w) = 2 + \left(\Delta \kappa \star  \frac{\chi_{E^\kappa}}{|E^\kappa|}\right)(w).
\end{equation}
The distribution $\Delta \kappa$, by \eqref{sod1} and \eqref{sod2}, is the sum of two terms 
\begin{equation}\label{lapk}
\Delta \kappa = \left(\int_{|\xi|=1} \langle \nabla \kappa (\xi), \nu(\xi) \rangle \, d\sigma(\xi)\right)\delta_0+ \pv \Delta \kappa,
\end{equation}
where $\nu(\xi)$ is the exterior unitary normal vector to the unit ball at $\xi$. 

Note that, since $\Delta \kappa$ is of class $C^1$ off the origin, by Lemma~\ref{intsingn}  $\Delta \kappa \star  \frac{\chi_{E^\kappa}}{|E^\kappa|}$ is a function of class $C^{0,\gamma}$ on  $\C\setminus E^\kappa$. Therefore, the limit in \eqref{oldjump} exists.

Finally, by \eqref{lapk} and \eqref{czb} we can take the number $\varepsilon_0$ in the statement of Theorem~\ref{te} so small that
$$
\left|	\Delta \kappa \star \frac{\chi_{E^\kappa}}{|E^\kappa|}(z) \right| \leq 1, 
\quad z \in \C \setminus \partial E^\kappa,
$$
which completes the proof of Lemma~\ref{lapposk}.
\end{proof}

\begin{remark}\label{Rem:CZ3}
The hypothesis on the smallness of the third order derivatives of $\kappa$ in \eqref{small} is used precisely in the last step of
 the proof of Lemma~\ref{lapposk}, to ensure that
 the kernel $\Delta \kappa$ provides a smooth Calder\'on-Zygmund operator ($T$ in the notation of Lemma~\ref{lapposk}) to which one can apply Lemma~\ref{intsingn}. Then, by \eqref{czb}, we have that $\Delta \kappa(z)$ is controlled by $\|\Delta \kappa\|_{\rm CZ}$, and by the definition  \eqref{CT}, $\|\Delta \kappa\|_{\rm CZ}$ can be controlled in terms of the second and third derivatives of $\kappa$.
\end{remark}

\subsubsection{The security region} \label{sec:security}
Recall that $P^\kappa$ is constant on $E^\kappa$ and thus the Hessian of $P^\kappa$ vanishes on $\mathring E^\kappa$. The Hessian of 
$P^\kappa$ has a jump at each point of the ellipse $\partial E^\kappa$. We define its value at $z\in \partial E^\kappa$  as
\begin{equation*}
 H(P^\kappa)(z) = \lim_{E^\kappa \not\ni w\rightarrow z} H(P^\kappa)(w).
\end{equation*}
Note that the limit above exists arguing as in the proof of Lemma~\ref{lapposk}.

We would like to find an expression for $H(P^\kappa)(z)$ at points $z \in \partial E^\kappa$ and for that we need first to prove tangential continuity
of the second order derivatives of $P^\kappa$. This follows from applying Lemma~\ref{tangcont} to first order derivatives of the potential \eqref{potell2}. 

More precisely, we have the following.

\begin{lemma}\label{hessian}
Let $E^\kappa$ be the ellipse satisfying the first Euler-Lagrange condition. Then any tangent vector at $z\in \partial E^\kappa$ is in the kernel of the symmetric operator  $H(P^\kappa)(z).$  
Consequently the unitary normal vector $\nu^\kappa(z)$  is an eigenvector of 
 $H(P^\kappa)(z)$ and the matrix of $H(P^\kappa)(z)$  in the basis $\{\nu^\kappa(z), \tau^\kappa(z)\}$ is of the form
  $$
  \left(
  \begin{matrix}
r^\kappa(z) & 0 \\
0& 0
\end{matrix}
\right)
$$
with $r^\kappa(z)\geq 1$. 
 \end{lemma}
 \begin{proof}
  By Lemma~\ref{tangcont} the Hessian $H(P^\kappa)$ is continuous at each point of the ellipse in the tangential direction, and being identically $0$ on the interior of the ellipse one concludes that $H(P^\kappa)(z)(\tau^\kappa(z))=0, \; z \in \partial E^\kappa$. 
  
Hence $\tau^\kappa(z)$ is an eigenvector of $H(P^\kappa)$, with eigenvalue zero. Being $H(P^\kappa)$ symmetric, we have that also $\nu^\kappa(z)$ is an eigenvector, and $H(P^\kappa)$ is diagonal in the basis $\{\nu^\kappa(z), \tau^\kappa(z)\}$. Let $r^\kappa(z)$ denote the eigenvalue corresponding to the eigenvector $\nu^\kappa(z)$; then $r^\kappa(z)$ is the limit of the Laplacian from the exterior of $E^\kappa$ at the point $z$, which satisfies the required estimate by Lemma~\ref{lapposk}.
 \end{proof}

We recall that by the first Euler-Lagrange equation we have $P^\kappa(z)=C^\kappa, \; z \in E^\kappa$. We now define the security region, an elliptical annulus of $E^\kappa$ in the exterior domain, where the second Euler-Lagrange condition is satisfied, even strictly. The idea is to prove that $P^\kappa$ is increasing in the direction of the outer normal $\nu^\kappa$, at least close to $\partial E^\kappa$.

Let $z \in \partial E^\kappa$, and define the function 
 $$
g^\kappa: t\in [0,+\infty) \mapsto P^\kappa(z+ t \nu^\kappa(z)). 
 $$
Note that $(g^\kappa)'(t ) = \langle \nabla P^\kappa(z+ t \nu^\kappa(z)),\nu^\kappa(z)\rangle$. Since $\nabla P^\kappa$ is a continuous function on $\C$ vanishing on $E^\kappa$ (again by the first Euler-Lagrange condition for $E^\kappa$), we have that $(g^\kappa)'(0)=0$. Moreover,  
$(g^\kappa)''(t) = \langle H(P^\kappa)(z+ t \nu^\kappa(z))\nu^\kappa(z),\nu^\kappa(z)\rangle$.  By \eqref{czholder} each second order derivative of $P^\kappa$ is continuous (in fact of class $C^{0,\gamma}$) up to the boundary in $\C \setminus E^\kappa$, and so $(g^\kappa)''(0) = \langle H(P^\kappa)(z)\nu^\kappa(z),\nu^\kappa(z)\rangle = \langle r^\kappa(z) \nu^\kappa(z),\nu^\kappa(z)\rangle = r^\kappa(z)$, where we have also used Lemma~\ref{hessian}.

Since $H(P^\kappa)$ is of class $C^{0,\gamma}$, we have that 
$$
|(g^\kappa)''(t) - (g^\kappa)''(0)| \leq c |t|^\gamma,
$$
where the constant $c$ is independent of $\kappa$ by Lemma~\ref{intsingn} and assumption \eqref{small}. In particular, 
$$
(g^\kappa)''(t) \geq (g^\kappa)''(0) - c\,|t|^\gamma = r^\kappa(z) - c\,|t|^\gamma \geq 1 - c\,|t|^\gamma,
$$
and so, there exists $\delta>0$, independent of $\kappa$, such that if $|t|<\delta$, then $(g^\kappa)''(t)>0$. This implies that $(g^\kappa)'$ is increasing in the interval $[0, \delta]$. Since $(g^\kappa)'(0)=0$, then $(g^\kappa)'(t)$ is positive close to zero, and so $g^\kappa$ is increasing close to zero. In other words, $P^\kappa$ is strictly increasing in a $\delta$-strip around $E^\kappa$. Let  
$$
N^\kappa = \{z \in \R^2 : \operatorname{dist}(z, E^\kappa) < \delta \};
$$
we call $N^\kappa\setminus E^\kappa$ the security region.  
We have proved that $P^\kappa$ is increasing in $N^\kappa\setminus E^\kappa$. Since $P^\kappa=C^\kappa$ in  $E^\kappa$, by the first Euler-Lagrange equation, we have that 
$P^\kappa> C^\kappa$ in $N^\kappa\setminus E^\kappa$.

In what follows, for brevity, we write $\kappa \rightarrow 0$ to mean that $\kappa$ tends to $0$ in the space $C^3(\mathbb{T})$, that is, the quantity
$$
\|\kappa\|_{C^3(\mathbb{T})}= \sup_{ j\in\{0,1,2, 3\}} \sup_{|\xi|=1} \left|\nabla^j\kappa (\xi)\right|
$$
becomes as small as we wish. 

\subsection{Approximating ellipses}\label{sect:approx} We now show that $E^\kappa$ and $P^\kappa$ are `close' to  $B(0,1)$ and $P^0$, respectively, for $\kappa$ small.

First of all, considering the system \eqref{systemaaa}, we conclude that 
\begin{equation*}
 a^\kappa \xrightarrow{\kappa \to 0} 1, \quad b^\kappa \xrightarrow{\kappa \to 0} 1.
 \end{equation*}
Moreover, the set $E^\kappa$ is close  to $B(0,1)$ in the Hausdorff distance for every $\kappa$ with $\|\kappa\|_{C^3(\mathbb{T})}$ small enough. In particular, this ensures that if $\|\kappa\|_{C^3(\mathbb{T})}$ is sufficiently small, then there exists a positive number $\gamma$ such that $B(0,1+\gamma) \subset N^\kappa$.

For the potentials, we have the following result.

\begin{lemma}\label{wconv}
$P^\kappa$ converges to $P^0$ uniformly on $\C \setminus B(0,1+\gamma)$, as $\kappa$ tends to $0$.
\end{lemma}
\begin{proof}
 We first estimate the terms involving $\kappa$.  We have
\begin{equation*}\label{kb}
\begin{split}
 &\left|\left(\kappa\star \frac{\chi_{E^\kappa}}{|E^\kappa|}\right)(z)-\left(\kappa\star \frac{\chi_{B(0,1)}}{|B(0,1)|}\right)(z)\right| \\*[7pt] &\le
\left|\frac{1}{|E^\kappa|}-\frac{1}{|B(0,1)|}\right| |(\kappa\star \chi_{E^\kappa})(z)| + \frac{1}{|B(0,1)|} 
\left|\left(\kappa\star \chi_{E^\kappa}-\kappa\star \chi_{B(0,1)}\right)(z)\right|  \\*[7pt]& \le
\left|\frac{1}{|E^\kappa|}-\frac{1}{|B(0,1)|}\right| \|\kappa\|_{\infty} |E^\kappa| +
\frac{\|\kappa\|_{\infty}}{|B(0,1)|} \left(|E^\kappa\setminus B(0,1)|+|B(0,1)\setminus E^\kappa|\right),
\end{split}
\end{equation*}
which tends to $0$ with $\kappa$. Let us deal now with the terms involving the logarithm. Let us remark that we can arrange things so that 
$B(0,1)\cup E^\kappa \subset B(0,1+\gamma) \subset N^\kappa$ ($\kappa$ close enough to $0$ in the $C^3$-norm). To estimate
\begin{equation*}
\left(\log|z|\star \frac{\chi_{E^\kappa}}{|E^\kappa|}\right)(z)-\left(\log|z|\star \frac{\chi_{B(0,1)}}{|B(0,1)|}\right)(z)
\end{equation*} 
we first note that the function above is harmonic in $\C \setminus (B(0,1)\cup E^\kappa)$ and vanishes at $\infty.$ Thus we only need
to estimate that difference for $z \in B(0,R)\setminus B(0,1+\gamma)$ for $R$ large. Now when $w \in B(0,1)\cup E^\kappa$ and 
$z \in  B(0,R)\setminus B(0,1+\gamma)$ the quantity $|\log|z-w||$ is bounded by a constant depending only on $R$ and the distance
between $B(0,1)\cup E^\kappa$ and $\R^2\setminus B(0,1+\gamma)$, which is positive. Hence we can argue as we did above in dealing with bounded kernels.
\end{proof}

Finally, we show that, for $\kappa$ small, the constants in the right-hand side of the Euler-Lagrange conditions are close to the constants for $\kappa=0$.

\begin{lemma}\label{convcons}
The constants $C^\kappa$ converge to $C^0$, as $\kappa$ tends to $0$.
\end{lemma}
\begin{proof}
Since $P^\kappa$ is constant on $E^\kappa$, we have in particular that 
\begin{equation*}
C^\kappa=P^\kappa(0) = \frac{1}{|E^\kappa|}\int_{E^\kappa} \left(-\log|z|+\kappa(z)\right)\,dz.
\end{equation*}
Remark that $\chi_{E^\kappa}(z)$ tends to  $\chi_{B(0,1)}(z)$ as $\kappa$ tends to $0$,
for each $z \notin \partial B(0,1)$, and apply the Dominated Convergence Theorem to conclude the proof.
\end{proof}

\subsubsection{Proof of the second Euler-Lagrange condition~\eqref{E2}} So far we have shown in Section~\ref{sec:security} that $P^\kappa(z) > C^\kappa$ for $z$ in the security region
$N^\kappa \setminus E^\kappa$. It remains to show that $P^\kappa(z) \geq C^\kappa$ outside $N^\kappa$. To this aim we use the approximation arguments in Section~\ref{sect:approx}, together with the fact that $P^0$ satisfies the first Euler-Lagrange condition 
$$
P^0(z)= C^0, \quad z \in B(0,1),
$$
and the second Euler-Lagrange condition in the strengthened form
\begin{equation}\label{seeucon0}
P^0(z)>C^0, \quad z \notin B(0,1),
\end{equation}
(see, e.g., the proof in \cite{CMMRSV} for $\alpha=0$).

Let $z \notin N^\kappa$; then since $B(0,1+\gamma) \subset N^\kappa$, we have that $z\notin B(0,1+\gamma)$,  and thus $P^0(z) \ge C^0+\eta,$ for some
positive $\eta,$ by \eqref{seeucon0}. If $\|\kappa\|_{C^3(\mathbb{T})}$ is sufficiently small, by applying  Lemmas~\ref{wconv}
and~\ref{convcons}, 
then for $\varepsilon = \eta/3$ we have 
\begin{equation*}
\begin{split}
P^\kappa(z)\ge C^0 + \eta -\varepsilon \geq C^\kappa+\eta - 2 \varepsilon = C^\kappa + \frac{\eta}{3} > C^\kappa, \quad z \notin N^\kappa,
\end{split}
\end{equation*}
which completes the proof of the second Euler-Lagrange condition.

\section{The higher-dimensional case}
In this section we briefly illustrate the higher-dimensional version of the perturbation result. Let $d\geq 3$, and let $I^\kappa$ denote the functional defined on probability measures $\mu \in \mathcal{P}(\R^d)$ as
\begin{equation}\label{end}
I^\kappa(\mu)=\int_{\R^d}\int_{\R^d} W^\kappa(x-y)\, d\mu(y) d\mu(x) + \int_{\R^d} |x|^2 \,d\mu(x),
\end{equation}
where the interaction potential $W^\kappa$ is given by
$$
W^\kappa(x) = \frac{1}{|x|^{d-2}}+\kappa(x), \quad x\in \R^d, \quad x\neq0,
$$
with $W^{\kappa}(0)=+\infty$, and $\kappa$ is an even real-valued function, homogeneous of degree $2-d$ and of class $C^3(\R^d\setminus\{0\}).$ For simplicity, in this section we assume that $\widehat{W^\kappa}>0$ outside the origin, and that $\kappa$ is even in each variable separately. 

The higher-dimensional version of Theorem~\ref{te}, under these slightly simplified assumptions, is the following.

\begin{theorem}\label{ted}Let $d\geq 3$. 
There exists $\ep_0>0$ such that if $\kappa$ is a real-valued function, homogeneous of degree $2-d$, even in each variable, of class $C^3$ off the origin, 
satisfies the smallness condition
\begin{equation}\label{smalld}
 |\nabla^j \kappa(x)| \le \ep_0 \quad \textrm{for } \, |x|=1 \quad \textrm{and }\, j \in \{0,1,2,3\},
\end{equation}
and $\widehat{W^\kappa}>0$ outside the origin, 
then there exists an ellipsoid with interior $E$, defined as in \eqref{ell}, such that the probability measure $\chi_E /|E|$ is the unique minimiser of the energy~\eqref{end}.
\end{theorem}

\begin{remark}
The assumption on the positivity of $\widehat{W^\kappa}$ outside the origin is not too restrictive and is considered only for the sake of simplicity.
Indeed, using standard properties of spherical harmonics one can prove that this condition is satisfied if $\kappa$ is assumed to be small enough in the $C^6$-norm on the sphere (see also \cite[p.~70]{S} and \cite[Lemma~6]{MOV2}).
\end{remark}

Existence of a compactly supported minimiser of  \eqref{end}, is straightforward under the assumptions of Theorem~\ref{ted}. Indeed, the overall potential 
$$
f(x,y)=W^\kappa(x-y)+\frac12(|x|^2+|y|^2)
$$
is lower semicontinuous, and is bounded from below due to homogeneity of $\kappa$ and \eqref{smalld}. Indeed, by homogeneity we have that $|\kappa(x)| \leq |x|^{2-d}  \sup_{|\xi|=1} |\kappa(\xi)|\leq \varepsilon_0 |x|^{2-d}$, and so $W^\kappa(x)\geq (1-\varepsilon_0)|x|^{2-d}>0$ if $\varepsilon_0<1$. In conclusion the energy $I^\kappa$ is lower semicontinuous and bounded from below by the confinement. This guarantees the existence of a compactly supported minimiser.

As for uniqueness, the assumption $\widehat{W^\kappa}(\xi)>0$ for $\xi\neq 0$ guarantees strict convexity of the energy. For strictly convex energies the unique minimiser is characterised by the Euler-Lagrange conditions, and in the following sections we show that they admit a unique ellipsoid as a solution.

We follow the strategy of Section~\ref{sect:EL}, and we only highlight the changes due to the higher-dimensional setting.

\subsection{The first Euler-Lagrange condition} The first Euler-Lagrange condition for $I^\kappa$ in \eqref{end} is 
\begin{equation}\label{E1d}
P^\kappa(\mu)(x) = C \quad \text{for }\mu\text{-a.e. } x \in \supp\mu,
\end{equation}
where the potential $P^\kappa$ of $\mu \in \mathcal{P}(\R^d)$ is defined as 
$$
P^\kappa(\mu)(x) = \left(\frac{1}{|\cdot|^{d-2}}\star \mu  \right) (x)+ (\kappa \star \mu)(x) + \frac{1}{2}|x|^2, \quad x\in \R^d.
$$
Note that since $\kappa$ is even in each variable, by uniqueness, the minimiser is symmetric with respect to all coordinate axes; in particular, if the minimiser is an ellipsoid, then it will be as in \eqref{ell}. As in the two-dimensional case, \eqref{E1d} is in fact equivalent to the vanishing of the Hessian of the potential in $E$, since the potential is even by assumption. We then focus on the system 
\begin{equation}\label{Hessd}
\partial_{ij}P^\kappa\left(\frac{\chi_E}{|E|}\right) = 0 \quad \text{in } E, \quad i,j=1,\dots,d,
\end{equation}
for which we want to exhibit a solution $E$, for $\kappa$ small. 

We start by evaluating the left-hand side of \eqref{Hessd} on a generic ellipsoid as in \eqref{ell}, by applying Lemma~\ref{constel} with kernels $H=\partial_{ij} W^0$, and $H=\partial_{ij} \kappa$, where we denoted $W^0=|\cdot|^{2-d}$. Note that \eqref{canc} is satisfied for both kernels. Lemma~\ref{constel} guarantees that $\partial_{ij}P^\kappa\left(\frac{\chi_E}{|E|}\right)$ is constant on $E$, and as in the two-dimensional case, we need to find an ellipsoid $E$ for which this constant is zero. 

By the assumption that $\kappa$ is even in each variable, the system \eqref{Hessd} simplifies greatly and reduces to the `diagonal' system 
\begin{equation}\label{Hessdiag}
\partial_{ii}P^\kappa\left(\frac{\chi_E}{|E|}\right) = 0, \quad \text{in}\quad E, \quad i=1,\dots,d,
\end{equation}
namely to a system of $d$ equations in $d$ unknowns, namely the semi-axes $a_1,\dots, a_d$ of the ellipsoid (see Remark~\ref{rem:sep-even}  for the case $d=2$). Indeed, for $i\neq j$ we have that the distributional derivatives of $W^0$ and $\kappa$ satisfy
$$
\partial_{ij} W^0 = \pv \partial_{ij} W^0,\quad \partial_{ij} \kappa = \pv \partial_{ij} \kappa,
$$
hence condition \eqref{Hessd}, for $i\neq j$, reduces to
\begin{equation}\label{space-oddity}
\pv \int_E \partial_{ij} \left(\frac{1}{|x|^{d-2}}\right) \,d x+ \pv \int_E \partial_{ij}\kappa (x) \,d x  =0,
\end{equation}
where we have used Lemma~\ref{constel} to replace convolution integrals with their evaluation at the origin, which belongs to $E$. 
Condition \eqref{space-oddity} is clearly satisfied since each term in the equation is the integral of an odd function in the variable $x_j$ on a symmetric domain, and hence is zero by Fubini's Theorem (see Remark \ref{rem:sep-even}).

The system \eqref{Hessdiag} can be made more explicit. First of all, for every $i$ we have that 
\begin{align*}
\partial_{ii} W^0 &= \left( \int_{|\xi|=1} \xi_i \,\partial_i  W^0(\xi) \, d\sigma(\xi)\right) \delta_0 + \pv \partial_{ii} W^0,\\
\partial_{ii} \kappa &= \left( \int_{|\xi|=1} \xi_i \,\partial_i  \kappa (\xi) \, d\sigma(\xi)\right) \delta_0 + \pv \partial_{ii} \kappa.
\end{align*}
Hence, for $x\in E$,
\begin{align*}
(\partial_{ii} W^0 \star \chi_E)(x) &=  \int_{|\xi|=1} \xi_i \,\partial_i  W^0(\xi) \, d\sigma(\xi)  -\int_{|\xi|=1} \log \big|\tfrac{\xi}{a}\big| \partial_{ii} W^0(\xi) \, d\sigma(\xi),\\
(\partial_{ii} \kappa\star \chi_E)(x) &=  \int_{|\xi|=1} \xi_i \,\partial_i  \kappa(\xi) \, d\sigma(\xi)  -\int_{|\xi|=1} \log  \big|\tfrac{\xi}{a}\big| \partial_{ii} \kappa(\xi) \, d\sigma(\xi),
\end{align*}
where we have used Lemma~\ref{constel} and the shorthand $\big|\tfrac{\xi}{a}\big|^2 = \frac{\xi_1^2}{a_1^2}+ \dots+\frac{\xi_d^2}{a_d^2}$. We now define 
$$
I_i(\kappa):= \int_{|\xi|=1} \xi_i \,\partial_i  \kappa(\xi) \, d\sigma(\xi), \quad \text{and } 
\quad F_i(a,\kappa):=  -\int_{|\xi|=1} \log  \big|\tfrac{\xi}{a}\big| \partial_{ii} \kappa(\xi) \, d\sigma(\xi),
$$
where $a=(a_1,\dots,a_d)$, and set 
\begin{align*}
G_i(a,\kappa):= &  \int_{|\xi|=1} \xi_i \,\partial_i  W^0 (\xi) \, d\sigma(\xi)  -\int_{|\xi|=1} \log \big|\tfrac{\xi}{a}\big| \partial_{ii} W^0(\xi) \, d\sigma(\xi) \\
& + I_i(\kappa) + F_i(a,\kappa) + \frac{\omega_d}{d} \prod_{j=1}^da_j,
\end{align*}
where $\omega_d$ is the surface measure of the unit sphere in $\R^d$. Then the system \eqref{Hessdiag} is equivalent to $G_i(a,\kappa)=0$ for $i=1,\dots,d$.

We want to show that, for $\kappa$ small as in Theorem~\ref{ted}, the system $G_i(a,\kappa)=0$ admits a solution `close' to the solution for $\kappa=0$, by using the Implicit Function Theorem. Note that in the expression of $G_i$ the kernel $\kappa$ and its derivatives only appear on the unit sphere, so the smallness assumption \eqref{smalld} on $\kappa$ is exactly what is needed there.
 Since the minimiser of $I^0$ (corresponding to $\kappa=0$) is the normalised characteristic function of the ball centred at zero with radius $(d-2)^{\frac1d}$, we have that ${G}_i((d-2)^{\frac1d},0)=0$ (where with an abuse of notation we used the shorthand $(d-2)^{\frac1d}$ to denote the vector in $\R^d$ with entries all given by $(d-2)^{\frac1d})$. 
 
 We now examine $G_i(a,\kappa)=0$ for $a$ close to $(d-2)^{\frac1d}$ and $\kappa$ close to zero. 
To apply the Implicit Function Theorem we need to show that the $(d\times d)$-matrix with $ij$-entry $\frac{\partial G_i}{\partial a_j}((d-2)^{\frac1d},0)$ is invertible. We have that 
$$
\frac{\partial{G}_i}{\partial a_j}((d-2)^{\frac1d},0)= (d-2)^{1-\frac1d}
\left(
d \int_{|\xi|=1} \xi_j^2 \xi_i^2\, d\sigma(\xi) - \int_{|\xi|=1} \xi_j^2\, d\sigma(\xi) + \frac{\omega_d}{d}
\right).
$$
Since $\int_{|\xi|=1} \xi_j^2 \,d\sigma(\xi)$ is independent of $j$ and 
$$
\sum_{j=1}^{d}\int_{|\xi|=1} \xi_j^2 \,d\sigma(\xi) = \omega_d, 
$$
we have that  $\int_{|\xi|=1} \xi_j^2 \,d\sigma(\xi) = \frac{\omega_d}{d}$, and so 
\begin{equation}\label{Gij}
\frac{\partial{G}_i}{\partial a_j}((d-2)^{\frac1d},0)= d(d-2)^{\frac{d-1}d}
\left(
\int_{|\xi|=1} \xi_j^2 \xi_i^2 \,d\sigma(\xi)
\right).
\end{equation}
We can easily see that the matrix with $ij$-entries as in \eqref{Gij} is positive semi-definite. Indeed, for $z\in \R^d$ we have 
$$
\sum_{i,j=1}^d \left(\frac{\partial{G}_i}{\partial a_j}((d-2)^{\frac1d},0)\right) z_iz_j = 
d(d-2)^{\frac{d-1}d}
\int_{|\xi|=1} \bigg(\sum_{j=1}^d z_j\xi_j^2\bigg)^2 d\sigma(\xi) \geq 0.
$$
On the other hand, if 
$$
\int_{|\xi|=1} \bigg(\sum_{j=1}^d z_j\xi_j^2\bigg)^2 d\sigma(\xi) = 0,
$$
then by continuity it must be
$$
\sum_{j=1}^d z_j\xi_j^2 =0 \quad \text{for all } \xi\in \R^d, |\xi|=1.
$$
Choosing $\xi = {e}_i$, where $e_i$ is the $i$-th coordinate vector, and varying $i=1,\dots,d$, we conclude that $z=0$, and hence that the matrix $\frac{\partial{G}_i}{\partial a_j}((d-2)^{\frac1d},0)$ is positive definite, and invertible. 
By the Implicit Function Theorem in Banach spaces (see, e.g., \cite[statement (10.2.1)]{D}) we can then conclude that there exists a unique solution $(a_1(\kappa),\dots, a_d(\kappa))$ with $a_i$ close to $(d-2)^{\frac1d}$, for $\kappa$ close to zero in the $C^2$-norm on $|\xi|=1$.

\subsection{The second Euler-Lagrange condition} Let $E^\kappa=E(a_1(\kappa),\dots,a_d(\kappa))$ be the unique solution of the Euler-Lagrange condition \eqref{E1d} found in the previous section. Here we prove that 
$$
P^\kappa(x) \geq C^\kappa \quad \text{for } x \in \R^d\setminus E^\kappa,
$$
where the potential $P^\kappa$ is defined as 
$$
P^\kappa(x):= \left(\left(\frac{1}{|x|^{d-2}}+\kappa(x)\right) \star \frac{\chi_{E^\kappa}}{|E^\kappa|} \right) (x)+ \frac{1}{2}|x|^2, \quad x
\in \R^d,
$$
and $P^\kappa(x)=C^\kappa$ in $E^\kappa$ by the first Euler-Lagrange condition.

The proof of the subharmonicity of the potential proceeds exactly as in the two-dimensional case treated in Section~\ref{sect:subh}, once it is shown the higher-dimensional equivalent of Lemma~\ref{lapposk}. This follows directly by the fact that for every $x \in \partial E^\kappa$ the limit of $\Delta P^\kappa(y)$ as $y\to x$, for $y\notin E^\kappa$, exists and satisfies the lower bound 
\begin{equation}\label{last}
\lim_{ E^\kappa \not \ni y \to x} \Delta P^\kappa(y) \ge  \frac d2, 
\end{equation}
provided $\kappa$ is small enough. To see this, note that for $y\notin E^\kappa$, 
$$
\Delta P^\kappa(y) = \left(\Delta \kappa \star \frac{\chi_{E^\kappa}}{|E^\kappa|} \right) (y)+ d,
$$
and hence by Remark~\ref{Rem:CZ3}, provided $\kappa$ is suitably small on the unit sphere, we can ensure that \eqref{last} is satisfied. Indeed, we can estimate the convolution with $\Delta \kappa$ with $\|\Delta \kappa\|_{\rm CZ}$, and by the definition  \eqref{CT}, $\|\Delta \kappa\|_{\rm CZ}$ can be controlled in terms of the second and third derivatives of $\kappa$ on the unit sphere. 

As for the approximation argument in Section~\ref{sect:approx}, we only need to ensure that the potentials are `close' outside a neighbourhood $B(0,(d-2)^{\frac1d}+\gamma)$ of $B(0,(d-2)^{\frac1d})$, namely of the minimiser of the functional $I^0$ with $\kappa=0$. We hence need the higher-dimensional version of Lemma~\ref{wconv}. Note that in this case the potential $\kappa$ is not bounded in $\R^d$, and hence the proof requires some modification. However, since by homogeneity, $|\kappa(x)| \leq |x|^{2-d}  \sup_{|\xi|=1} |\kappa(\xi)|$, which is small if $x\notin B(0,R)$, with $R$ large, we can reduce to proving convergence in $B(0,R)\setminus B(0,(d-2)^{\frac1d}+\gamma)$, where all potentials are bounded.


\bigskip\bigskip

\noindent
\textbf{Acknowledgements.}
JM and JV acknowledge support from the grants 2017-SGR-395 (Generalitat de Catalunya), PID2020-112881GB-I00 and 
Severo Ochoa and Maria de Maeztu CEX2020-001084-M.
MGM acknowledges support by MIUR--PRIN 2017. MGM and LR are members of GNAMPA--INdAM. 
LS acknowledges support by the EPSRC under the grants EP/V00204X/1 and EP/V008897/1.


\end{document}